\documentclass{amsart}

\usepackage{verbatim}

\usepackage{amsfonts}

\usepackage{amsthm}

\usepackage{amssymb}

\usepackage{amsmath}

\usepackage{mathabx}

\usepackage{enumerate}

\usepackage[all]{xy}

\usepackage{graphicx}

\usepackage{hyperref}

\newcommand{\N}{\mathbb{N}}

\newcommand{\Z}{\mathbb{Z}}

\newcommand{\R}{\mathbb{R}}

\newcommand{\C}{\mathbb{C}}

\newcommand{\Manoa}{M\=anoa}

\newcommand{\Hawaii}{Hawai\kern.05em`\kern.05em\relax i}

\newcommand{\band}{\C[X;E]}

\newcommand{\AX}{\mathcal{A}^p_E(X)}

\newcommand{\AXro}{\mathcal{A}^p_E(X)^{\$,\omega}}
\newcommand{\AXrotwo}{\mathcal{A}^2_E(X)^{\$,\omega}}

\newcommand{\AXr}{\mathcal{A}^p_E(X)^{\$}}

\newcommand{\AXtwoC}{\mathcal{A}^2_{\C}(X)}
\newcommand{\AXtwoE}{\mathcal{A}^2_{E}(X)}

\newcommand{\KX}{\mathcal{K}^p_E(X)}

\newcommand{\pcom}{$\mathcal{P}$-compact}

\newcommand{\pfred}{$\mathcal{P}$-Fredholm}

\usepackage{color}

\DeclareMathOperator{\diam}{diam}
\DeclareMathOperator{\supp}{supp}
\DeclareMathOperator{\prop}{prop}
\DeclareMathOperator{\opspec}{\sigma_{op}}
\DeclareMathOperator{\opspecRRS}{\sigma_{op}^{RRS}}
\DeclareMathOperator{\opspecRoe}{\sigma_{op}^{Roe}}
\DeclareMathOperator{\kernel}{kernel}
\def\lp{\ell^p_E}

\usepackage{paralist}

\theoremstyle{plain}
\newtheorem{theorem}{Theorem}[section]
\newtheorem{lemma}[theorem]{Lemma}
\newtheorem{corollary}[theorem]{Corollary}
\newtheorem{proposition}[theorem]{Proposition}

\newtheorem{definition-theorem}[theorem]{Definition / Theorem}
\newtheorem{claim}[theorem]{Claim}

\newtheorem*{conjecture*}{Conjecture}
\newtheorem*{theorem*}{Theorem}
\newtheorem*{claim*}{Claim}

\theoremstyle{definition}
\newtheorem{definition}[theorem]{Definition}

\newtheorem{examples}[theorem]{Examples}

\theoremstyle{remark}
\newtheorem{remark}[theorem]{Remark}

\newtheorem*{example*}{Example}  
\newtheorem*{remark*}{Remark}

\begin{document}

\title{A metric approach to limit operators}

\author{J\'{a}n \v{S}pakula}
\address{Mathematical Sciences, University of Southampton, SO17 1BJ, United Kingdom}
\email{jan.spakula@soton.ac.uk}
\author{Rufus Willett}
\address{2565 McCarthy Mall, University of \Hawaii\ at \Manoa, Honolulu, HI 96822, USA}
\email{rufus.willett@hawaii.edu}
\thanks{Second author partially supported by the US NSF}


\date{11 September 2014}

\begin{abstract}
We extend the limit operator machinery of Rabinovich, Roch, and Silbermann from $\Z^N$ to (bounded geometry, strongly) discrete metric spaces.  We do not assume the presence of any group structure or action on our metric spaces.  Using this machinery and recent ideas of Lindner and Seidel, we show that if a metric space $X$ has Yu's property A, then a band-dominated operator on $X$ is Fredholm if and only if all of its limit operators are invertible.  We also show that this always fails for metric spaces without property A.
\end{abstract}

\maketitle

\section{Introduction}

Thinking of an operator $A$ on $\ell^p(\Z^N)$ as a $\Z^N$-by-$\Z^N$ matrix, we say that $A$ is a \emph{band operator} if the only non-zero entries in its matrix appear within a fixed distance from the diagonal. The \emph{band-dominated} operators in $\ell^p(\Z^N)$ are then norm--limits of band operators.  The group $\Z^N$ acts on $\ell^p(\Z^N)$ by shifts: for each $m\in \Z^N$ there is an isometric isomorphism $V_m:\ell^p(\Z^N)\to \ell^p(\Z^N)$ defined on the canonical basis $\{\delta_k\}_{k\in \Z^N}$ by
$$
V_m:\delta_k\mapsto \delta_{k+m}.
$$
Given a band-dominated operator $A$ and a sequence $(m_n)_{n\in\N}$ in $\Z^N$ converging to infinity, the sequence $(V_{-m_n}AV_{m_n})_{n\in \N}$ of \emph{shifts} of $A$ by $m_n$ always contains a strongly convergent subsequence, and the strong limit is called the \emph{limit operator} of $A$ associated with the given subsequence. The collection of all limit operators of $A$ is called the \emph{operator spectrum} and is denoted $\opspec(A)$.  See the book \cite{Rabinovich:2004qy} and the paper \cite{Lindner:2014aa} for a recent list of relevant references, and also for many examples of band-dominated operators, limit operators, and their applications. 

One of the most important goals of limit operator theory is to study the Fredholm property for the class of band-dominated operators on $\ell^p$-spaces over $\Z^N$ in terms of the operator spectrum.   The following theorem of Rabinovich, Roch and Silbermann characterises when band-dominated operators are Fredholm.

\begin{theorem}\cite[Theorem 2.2.1]{Rabinovich:2004qy}\label{thm:fred}
Let $T$ be a band-dominated operator in $\ell^p(\Z^n)$. Then $T$ is Fredholm if and only if all $S\in\opspec(T)$ are invertible and their inverses are uniformly bounded in norm. \qed
\end{theorem}
\noindent Whether the uniform boundedness condition in the above is really necessary was a long-standing open problem: it was recently shown not to be by Lindner and Seidel \cite{Lindner:2014aa}.

John Roe \cite{Roe:2005lr} has explained the connection between the above setup and the large scale (`coarse') geometry of more general discrete groups.  In the Hilbert space case (i.e.~$p=2$), coarse geometers call the band operators \emph{finite propagation operators} and the collection of all band-dominated operators comprises the \emph{translation C*-algebra} (also called the \emph{uniform Roe algebra} in the literature). Roe extended the symbol calculus implicit in Theorem \ref{thm:fred} to all discrete groups $\Gamma$ and proved the Fredholmness criterion \ref{thm:fred} for all \emph{exact} discrete groups $\Gamma$.


Summarising, Roe established that the limit operator theory setup is inherently \emph{coarse geometric} in nature, and that one may expect that the operator theoretic properties of band-dominated operators on a discrete group are closely related to the large--scale geometry of the underlying discrete group.\\

Having this philosophy in mind, we extend the framework of limit operator theory to a purely metric setting: we consider band-dominated operators over an arbitrary (strongly discrete, bounded geometry) metric space $X$. As well as substantially generalising existing results in the literature, we believe our approach clarifies the geometric inputs that are implicitly used in the $\Z^N$ case.  

The traditional setting of limit operator theory and band-dominated operators (on $\Z^N$ or discrete groups) pertains to many naturally occurring operators: for example geometric differential operators on universal covers of compact manifolds, and their discretizations.  See e.g. \cite[Introduction]{Lindner:2014aa} for a recent survey and a collection of references. Our general metric setting covers similar operators on discretizations of general open manifolds satisfying some reasonable `bounded geometry' conditions (for example, bounded sectional curvature and injectivity radius bounded below).

We prove that an analogue of Theorem \ref{thm:fred} holds in this metric setting, including removing the uniform boundedness condition, provided that the space $X$ has Yu's \emph{Property A} \cite[Section 2]{Yu:200ve}. Note that Property A is equivalent to exactness in the case when $X$ is a discrete group, and thus we recover Roe's theorem as a special case.  We also show that our Fredholmness criterion \emph{always} fails for spaces without Property A, and thus our results are in some sense best possible; this is related to the existence of non-compact \emph{ghost operators}. Finally, in the Hilbert space case, we explain that our framework can be phrased in terms of coarse groupoids and groupoid C*-algebras.\\

Let us outline the main ideas of our setup in general terms. The notions of band and band-dominated operators make sense for any (strongly discrete, bounded geometry) metric space $X$, as the definitions only use the metric on $X$. Indeed, band-dominated operators on such spaces have been extensively studied in the context of the coarse Baum--Connes conjecture. The first obstacle to be overcome is to generalise the notion of limit operator.
In the case of $\Z^N$, or generally a discrete group, limit operators are strong limits of (a subsequence of) shifts of a given operators; these are not available for a general space $X$. We propose the following construction.

Let $X$ be a discrete metric space and $A$ be a band-dominated operator on $\ell^p(X)$. Instead of sequences of points of $X$ tending to infinity, we shall follow Roe and associate limit operators of $A$ to the points $\omega$ of the Stone--\v{C}ech boundary of $X$, $\partial X=\beta X\setminus X$.  For each $\omega\in\partial X$ we construct a canonical \emph{limit space}, denoted $X(\omega)$, which captures the geometry of $X$ as one `looks towards infinity, in the direction of $\omega$'. The limit operator, $\Phi_\omega(A)$, of $A$ associated to $\omega$, will be a bounded operator on $\ell^p(X(\omega))$. We note that when $X$ is a discrete group, all limit spaces are (canonically isometric to) $X$ itself, and our notion of limit operators agrees with the original one.  

Our construction is related to that of Georgescu \cite{Georgescu:2011cr}, which also makes explicit use of ghost operators and property A.  Our construction has the advantages over Georgescu's that our limit operators are perhaps more concretely described, and that it works for operators on a large class of $\ell^p$ spaces, rather than just on Hilbert spaces; on the other hand, Georgescu's construction is more general than ours in that it works for non-discrete metric spaces.  We note that the techniques used in this paper and those used in \cite{Georgescu:2011cr} are quite different.

\subsection*{Outline of the paper}

After clarifying the details of the above construction in Sections \ref{sec:limitspaces} and \ref{sec:limitops}, we set out to prove the Fredholmness criterion, Theorem \ref{main}, our analogue of Theorem \ref{thm:fred}. The proof of the `easy' implication, which holds without any extra assumptions on the space $X$, is given in Section \ref{sec:mainthm}. The proof of another implication occupies Section \ref{sec:parametrices}; this one requires the assumption of Property A for $X$.  Similar results exist in the literature, but they use (complete) positivity in the Hilbert space context; the approach in this paper works in the general $\ell^p$-setting, but requires different, somewhat more technical arguments. Section \ref{sec:unifbdd} removes the uniform boundedness requirement in the Fredholmness criterion: we generalise the proof of Lindner and Seidel \cite{Lindner:2014aa} for $\Z^N$ and explain that their `main tool' (proved for $\Z^N$ directly) is again Property A in disguise. In Section \ref{sec:ghosts} we show necessity of Property A for Theorem \ref{main}: this amounts to showing that the observation of Roe \cite{Roe:2005lr} on ghost operators and symbol calculus works in our general setting. Finally, Appendix \ref{ultrafilterssec} collects the conventions on ultrafilters that we use in the paper, Appendix \ref{sec:comparison} compares our approach to limit operators to others in the literature, and Appendix \ref{sec:groupoids} outlines the alternative picture of our setup using coarse groupoids and their C*-algebras in the Hilbert space case.

\subsection*{Acknowledgements}

The second author would like to thank Martin Finn-Sell for some very useful conversations about ultralimits and `geometry at infinity', and John Roe for many relevant conversations on the theory of band-dominated (or finite propagation) operators.   He would also like to thank Jerry Kaminker for organising a mini-workshop connected to the subject of this paper, and the participants of this workshop for several useful comments.  Finally, he would like to thank the University of Southampton for its hospitality during part of the work on this paper.

\section{Preliminaries}\label{sec:preliminaries}

We will work with operators associated to metric spaces as in the following definition.

\begin{definition}\label{sdbg}
Let $(X,d)$ be a metric space.  For $x\in X$ and $r>0$ we denote by
$$
B_X(x;r):=\{y\in X~|~d(x,y)\leq r\}
$$
the closed ball about $x$ of radius $r$ (we will often drop the subscript `$_X$' if there is no ambiguity). 

A metric space $(X,d)$ is \emph{strongly discrete} if the collection
$$
\{d(x,y)\in \R~|~x,y\in X\}
$$
of values of the metric is a discrete subset of $\R$.  It is \emph{bounded geometry} if for every $r>0$ there exists $N=N(r)\in \N$ such that $|B(x;r)|\leq N$ for all $x\in X$.

We say `$X$ is a space' as shorthand for `$X$ is a strongly discrete, bounded geometry metric space' throughout the rest of this paper.
\end{definition}

We make the blanket assumption of strong discreteness for reasons of simplicity: removing it would make many of the arguments below significantly more technical.  Moreover, if $(X,d)$ is any bounded geometry metric space and we set $d'(x,y):=\lceil d(x,y)\rceil$ (here $\lceil\cdot \rceil$ is the ceiling function), then the metric space $(X,d')$ is a space in our sense.  The results of this paper apply directly to $(X,d')$, and as the metrics $d$ and $d'$ are \emph{coarsely equivalent} (roughly, give rise to the same large-scale geometry), they can easily be transferred back to the original metric space $(X,d)$; thus the strong discreteness  assumption does not really lose generality.  The assumption of bounded geometry, on the other hand, is substantial: it seems most of the results of this paper fail without it. 

Note that the class of strongly discrete, bounded geometry metric spaces in particular includes countable discrete groups (endowed with left invariant proper metrics, for instance a word metric if the group is finitely generated).

If $E$ is a Banach space, denote by $\mathcal{L}(E)$ the Banach algebra of bounded linear operators on $E$.  The following `matrix algebras' our are main object of study.

\begin{definition}\label{band}
Let $X$ be a space and $E$ a Banach space.  Let $A=(A_{xy})_{x,y\in X}$ be an $X$-by-$X$ indexed matrix with values in $\mathcal{L}(E)$.  The matrix $A$ is a \emph{band operator (on $X$, with coefficients in $\mathcal{L}(E)$)} if
\begin{enumerate}
\item the norms $\|A_{xy}\|$ are uniformly bounded;
\item the \emph{propagation} of $A$ defined by
$$
\prop(A):=\sup\{d(x,y)~|~A_{xy}\neq 0\}
$$
is finite.
\end{enumerate}
Let $\band$ denote the collection of all band operators on $X$ with coefficients in $E$; the bounded geometry condition on $X$ implies that the usual matrix operations and the algebra structure of $\mathcal{L}(E)$ make $\band$ into an algebra.
\end{definition}

\begin{examples}\label{bandex}
Let $X$ be a space, and $E$ a Banach space.  The following two classes of operators are the basic examples of band operators.
\begin{enumerate}
\item Let $f:X\to \mathcal{L}(E)$ be a bounded function (in other words, an element of $l^\infty(X,\mathcal{L}(E))$. Then the diagonal matrix defined by 
$$
A_{xy}=\begin{cases}f(x) & x=y \\ 0 & \text{otherwise} \end{cases}
$$
is a band-operator of propagation $0$. We refer to these as \emph{multiplication operators}, as they act as such in the natural representation (described in Corollary \ref{boundrep} below).

An important special case occurs when $f$ is just a bounded complex-valued function on $X$, identified with the corresponding function on $X$ with values in the scalar multiples of the identity operator $1_E\in \mathcal{L}(E)$.  We identify scalar-valued functions on $X$ with elements of $\band$ in this way without further comment.  In particular, if $Y$ is a subset of $X$, we denote by $P_Y$ the idempotent element of $\band$ corresponding to the characteristic function of $Y$. Note that the operator $P_{\{x\}}AP_{\{y\}}$ identifies naturally with the matrix entry $A_{xy}$.
\item Let $D,R$ be subsets of $X$, and let $t:D\to R$ be a bijection such that $\sup_{x\in D} d(x,t(x))$ is finite (such a function $t$ is called a \emph{partial translation} on $X$).  Define an $X$-by-$X$ indexed matrix by
$$
V_{yx}=
\begin{cases}1_E & x\in D \text{ and } y=t(x)\\
0 & \text{otherwise.} \end{cases}
$$
Then $V$ is a band-operator, called a \emph{partial translation operator}. 
\end{enumerate}
\end{examples}

In fact the two classes of operators above generate $\band$ as an algebra in a precise sense.  Versions of the following lemma are very well-known.

\begin{lemma}\label{decomp}
Let $A$ be an element of $\band$ with propagation at most $r$.  Let $N=\sup_{x\in X}|B(x;r)|$.  Then there exist multiplication operators $f_1,...,f_N\in l^\infty(X,\mathcal{L}(E))$ such that $\|f_k\|\leq \sup_{x,y}\|A_{xy}\|$ for $k\in \{1,...,N\}$ and partial translation operators $V_1,...,V_N$ of propagation at most $r$ such that  
$$
A=\sum_{k=1}^N f_kV_k.
$$
\end{lemma}

\begin{proof}
Inductively define partial translations $t_1,t_2,...$ as follows.  Let $t_0$ be the empty partial translation.  Having defined $t_1,...,t_k$, let $t_{k+1}$ be any partial translation such that $d(x,t_{k+1}(x))\leq r$ for all $x$ in the domain of $t_{k+1}$, such that the graph of $t_{k+1}$ is disjoint from those of $t_1,...,t_k$, and such that the graph of $t_{k+1}$ is maximal with respect to these conditions.  We claim that $t_k$ is empty for all $k>N$.  Indeed, if not, then there exists $k>N$ and a point $x$ in the domain of $t_k$.  Then maximality of $t_1,...,t_{k-1}$ implies that $x$ is in the domain of all of these partial translations and so 
$$
t_1(x),....,t_k(x)
$$
are distinct points in $B(x;r)$, which contradicts the definition of $N$.  

For $k=1,...,N$, then, set $V_k$ to the partial translation operator corresponding to $t_k$, and define 
$$
f_k:X\to \mathcal{L}(E),~~~f_k(x)=\left\{\begin{array}{ll} A_{t_k(x)x} & x\in \text{domain}(t_k) \\ 0 & \text{otherwise} \end{array}\right..
$$
It is not difficult to check that these operators have the desired properties.
\end{proof}

If $X$ is a space, $E$ a Banach space and $p\in(1,\infty)$, we shall use the notation $\ell^p_E(X):=\ell^p(X,E)$ for the Banach space of $p$-summable functions from $X$ to $E$.

\begin{corollary}\label{boundrep}
Let $X$ be a space, $E$ a Banach space, and $p$ a number in $(1,\infty)$. Let $A\in \band$ have propagation at most $r$.  

Then the the operator on $\ell^p_E(X)$ defined by matrix multiplication by $A$ is bounded, with norm at most 
$$
\sup_{x,y}\|A_{xy}\|\cdot \sup_{x\in X}|B(x;r)|.
$$
\end{corollary}  

\begin{proof}
Writing $A$ as in Lemma \ref{decomp}, the operators $f_k$ have norm at most $\sup_{x,y}\|A_{xy}\|$ and the operators $V_k$ have norm one (or zero, if the corresponding partial translation has empty domain).
\end{proof}

Now, using the above corollary we may represent $\band$ by bounded operators on each Banach space $\ell^p_E(X)$ by matrix multiplication.  It is easy to see this representation is faithful, and we will usually identify $\C[X;E]$ with its image in $\mathcal{L}(\ell^p_E(X))$ in what follows. 

\begin{definition}\label{banddominated}
Let $X$ be a space, $E$ a Banach space, and $p$ a number in $(1,\infty)$.  The closure of $\band$ in its matricial representation on $\ell^p_E(X)$ is denoted $\AX$.  Elements of $\AX$ are called \emph{band-dominated} operators on $\ell^p_E(X)$.
\end{definition}

\begin{remark}
If $E$ is a Hilbert space (and $p=2$), then the Banach algebra $\AXtwoE$ is in fact a C*-algebra. Moreover if $E=\C$, then this C*-algebra is usually called the \emph{translation C*-algebra} or the \emph{uniform Roe algebra} of $X$ in the fields of coarse geometry and operator algebras, and denoted $C^*_u(X)$.  

On the other hand, If $E$ is an infinite dimensional Hilbert space, then the closure of the band operators in on $\ell^2_E(X)$ \textsl{all of whose matrix entries are compact operators} is called the \emph{Roe algebra} of $X$ in these areas, and denoted $C^*(X)$.  Hence in this case $\AXtwoE$ contains, but is strictly larger than, the Roe algebra of $X$. 
\end{remark}

\begin{definition}\label{Xcompact}
Let $X$ be a space, $E$ a Banach space, and $p$ a number in $(1,\infty)$.  Recall that if $Y$ is a subset of $X$, then $P_Y$ denotes the (norm one) idempotent operator on $\ell^p_E(X)$ corresponding to the characteristic function of $Y$.

A bounded operator $K$ on $\ell^p_E(X)$ is \emph{\pcom}~if for any $\epsilon>0$ there exists a finite subset $F$ of $X$ such that 
$$
\|K-KP_F\|<\epsilon \text{ and } \|K-P_FK\|<\epsilon.
$$
Write $\KX$ for the collection of all \pcom~operators on $\ell^p_E(X)$.
\end{definition}

\begin{remark}\label{pcomrem}
If $E$ is finite dimensional, the \pcom~operators on $\ell^p_E(X)$ are exactly those that can be approximated in norm by finite rank operators.  Any such $\ell^p_E(X)$ has a Schauder basis, and so has the approximation property.  Thus the \pcom~operators are exactly the compact operators in this case.  Many of the results of this paper are easier to digest (but still non-trivial) in the case that $E$ is finite dimensional, or even when $E=\C$, and the reader is encouraged to consider this case.
\end{remark}

\begin{lemma}\label{comideal}
The collection $\KX$ is a closed two-sided ideal in the algebra $\AX$ of band-dominated operators.  
\end{lemma}

\begin{proof}
Let $K$ be \pcom.  If $F\subseteq X$ is finite and such that $\|P_FK-K\|<\epsilon$ and $\|KP_F-K\|<\epsilon$, then $\|P_FKP_F-K\|<2\epsilon$.  As each operator $P_FKP_F$ is a band operator, this shows that $K$ is band-dominated.   The collection of \pcom~operators is norm closed as the norm of any $P_F$ (where $F$ is non-empty) is one.  Finally, note that if $A$ is a band-operator, then for any finite subset $F$ of $X$, there exists a finite subset $G$ of $X$ such that $P_FAP_G=P_FA$ and $P_GAP_F=AP_F$: indeed, we may just take
$$
G=\{x\in X~|~d(x,F)\leq \prop(A)\}.
$$
The fact that $\KX$ is an ideal in $\AX$ follows from this.  
\end{proof}

\begin{definition}\label{Xfredholm}
Let $X$ be a space, $E$ a Banach space, and $p$ a number in $(1,\infty)$.  A band-dominated operator $A$ on $\ell^p_E(X)$ is \emph{\pfred}~if there exists a bounded operator $B$ on $\ell^p_E(X)$ such that $AB-1$ and $BA-1$ are in $\KX$, i.e.\ $A$ is invertible modulo the \pcom~operators.
\end{definition}

\begin{remark}
If $E$ is finite dimensional then Remark \ref{pcomrem} and Atkinson's theorem together imply that the \pfred~operators are precisely the band-dominated operators that are Fredholm in the usual sense.
\end{remark}

The central goal of this paper is to derive a criterion determining when a band-dominated operator is \pfred. It turns out this is intimately connected to the geometry at infinity of $X$: in the next section, we will discuss the necessary preliminaries from metric space theory.

\section{Limit spaces}\label{sec:limitspaces}

Throughout this section, $X$ is a space in the sense of Definition \ref{sdbg}.  We will freely use the terminology of ultrafilters on $X$ and the associated Stone-\v{C}ech compactification $\beta X$ and boundary $\partial X$: this material is recalled for the reader's convenience in Appendix \ref{ultrafilterssec}.

The following definition has already appeared in Example \ref{bandex} above, but we isolate it here as it is particularly important for this section. 

\begin{definition}\label{pt}
A function $t:D\to R$ with domain and range subsets $D$, $R$ of $X$ is called a \emph{partial translation (on $X$)} if it is a bijection from $D$ to $R$, and if 
$$
\sup_{x\in D}d(x,t(x))
$$
is finite.  
\end{definition}

\begin{definition}\label{comp}
Fix an ultrafilter $\omega\in \beta X$.  A partial translation $t:D\to R$ on $X$ is \emph{compatible} with $\omega$ if $\omega(D)=1$ (i.e.\ $\omega$ is in the closure $\overline{D}$).  If $\omega$ is compatible with $t$, then considering $t$ as a function $t:D\to \beta X$ we may use Definition \ref{limdef} to define
$$
t(\omega):=\lim_\omega t\in \beta X.
$$

For a fixed ultrafilter $\omega\in \beta X$, an ultrafilter $\alpha\in\beta X$ is \emph{compatible} with $\omega$ if there exists a partial translation $t$ which is compatible with $\omega$, and which satisfies $t(\omega)=\alpha$.
\end{definition}

\begin{remark}\label{rem:comp-alt}
If we unravel the definition of $\lim_\omega$ above, we arrive at the following alternative description of ultrafilters $\alpha$ compatible with a given $\omega$.

An ultrafilter $\alpha\in \beta X$ is compatible with $\omega\in\beta X$, if there exists a partial translation $t:D\to R$ such that $\omega(D)=1$ and such that for any $S\subseteq X$ we have $\omega(S)=1$ iff $\alpha(t^{-1}(S\cap R))=1$.

It is easy to see from this that the relation of compatibility on elements of $\beta X$ is symmetric and reflexive.  If moreover $s:D_s\to R_s$ shows $\alpha$ compatible to $\beta$, and $t:D_t\to R_t$ shows that $\beta$ is compatible to $\gamma$, then $\beta(D_t)=\beta(R_s)=1$ and so $\beta(D_t\cap R_s)=1$; hence $\alpha(s^{-1}(D_t\cap R_s))=1$ and thus
$$
t\circ s|_{s^{-1}(D_t\cap R_s)}:s^{-1}(D_t\cap R_s)\to t(R_s\cap D_t)
$$ 
shows that $\alpha$ is compatible to $\gamma$.  Thus compatibility is an equivalence relation.  

Note that if $\omega\in X\subseteq \beta X$ (i.e. $\omega$ is a principal ultrafilter), then the collection of ultrafilters compatible with $\omega$ consists precisely of all principal ultrafilters on $X$, i.e.\ all points of $X$ itself.
\end{remark}

We will now show an essential uniqueness statement: if $\omega$ is an ultrafilter then any two partial translations $s,t$ that are compatible with $\omega$ and such that $s(\omega)=t(\omega)$ are essentially the same, where `essentially' means `off a set of $\omega$-measure zero'.

\begin{lemma}\label{esssame}
Let $X$ be a space, and $\omega$ be an ultrafilter on $X$.  Say $t:D_t\to R_t$ and $s:D_s\to R_s$ are two partial translations compatible with $\omega$ such that $s(\omega)=t(\omega)$.  Then if
$$
D:=\{x\in D_t\cap D_s~|~t(x)=s(x)\}
$$
we have that $\omega(D)=1$.
\end{lemma}

We first need a combinatorial lemma (which is probably very well-known).

\begin{lemma}\label{dislem}
Let $B$ and $C$ be sets.  Let $s,t:B\to C$ be bijections such that for all $a\in B$, $s(a)\neq t(a)$.  Then there exists a decomposition of $B$ into three disjoint subsets
$$
B=B_0\sqcup B_1\sqcup B_2
$$
such that for all $i\in \{0,1,2\}$, $s(B_i)\cap t(B_i)=\varnothing$.
\end{lemma}

The case that $B=C=\{1,2,3\}$, $s$ is the identity, and $t$ is a cyclic permutation, shows that one cannot get away with less than three subsets.

\begin{proof}
Replacing $t$ and $s$ with $s^{-1}\circ t$ and $s^{-1}\circ s$, it suffices to show that if $B$ is a set and $t:B\to B$ a bijection such that $t(b)\neq b$ for all $b\in B$, then there exists a decomposition $B=B_1\sqcup B_2\sqcup B_3$ such that $t(B_i)\cap B_i=\varnothing$ for all $i\in \{0,1,2\}$.  We will now prove this. 

As $t$ is now a bijection from $B$ to itself, it gives rise to an action of $\Z$ (thought of as generated by $t$) on $B$ which partitions $B$ into orbits.  As $t(b)\neq b$ for all $b\in B$ there are no single point orbits, and so each orbit has one of the following forms.
\begin{enumerate}[(1)]
\item $\{...,t^{-2}(b),t^{-1}(b),b=t^0(b),t(b),t^2(b),...\}$ (going on infinitely in both directions) for some $b\in B$.
\item $\{b=t^0(b),t(b),...,t^n(b)\}$ for some $n\geq 1$ and $b\in B$ such that $t^{n+1}(b)=b$.
\end{enumerate}
Define the subsets $B_0$, $B_1$ and $B_2$ as follows.  For each orbit, fix once and for all a representation of the type above.  For an orbit of type (2) with $n$ even and $i=n$, put $t^i(b)$ into $B_2$.  In all other cases, put $t^i(b)$ into $B_{i\text{ mod } 2}$ (where $i \text{ mod }2$ is always construed as $0$ or $1$).  A routine case-by-case analysis shows that this works.
\end{proof}

\begin{proof}[Proof of Lemma \ref{esssame}]
Define
$$
C:=(D_s\cap D_t)\setminus D=\{x\in D_s\cap D_t~|~t(x)\neq s(x)\}.
$$  
Noting that as $\omega(D_s)=\omega(D_t)=1$, we have $\omega(D_s\cap D_t)=1$.  Hence if we assume for contradiction that $\omega(D)=0$, then $\omega(C)=1$. Lemma \ref{dislem} implies that we may decompose $C$ into three disjoint subsets, $C=C_0\sqcup C_1\sqcup C_2$, such that for $i\in\{1,2,3\}$, 
\begin{equation}\label{disjoint}
t(C_i)\cap s(C_i)=\varnothing.
\end{equation}
We must have $\omega(C_i)=1$ for some $i\in \{0,1,2\}$; say without loss of generality $\omega(C_1)=1$.  Write $\alpha=\lim_\omega t$.  Then 
$$
\alpha=\lim_\omega s|_{C_1}=\lim_\omega t|_{C_1},
$$
whence, by definition of $\omega$-limits, $\alpha$ is in the closures of both $t(C_1)$ and $s(C_1)$, i.e.\
$$
\alpha(t(C_1))=\alpha(s(C_1))=1.
$$
This contradicts line \eqref{disjoint}, so $\omega(D)=1$ as claimed.  
\end{proof}

\begin{definition}\label{comfam}
Fix an ultrafilter $\omega$ on $X$.  Write $X(\omega)$ for the collection of all ultrafilters on $X$ that are compatible with $\omega$.  

A \emph{compatible family} for $\omega$ is a collection of partial translations $\{t_\alpha\}_{\alpha\in X(\omega)}$ indexed by $X(\omega)$ such that each $t_\alpha$ is compatible with $\omega$ and satisfies $t_\alpha(\omega)=\alpha$.   
\end{definition}

The set $X(\omega)$ should be thought of as the collection of ultrafilters that are at a `finite distance' from $\omega$. Our next goal is to make this more precise by equipping $X(\omega)$ with a canonical metric.

\begin{proposition}\label{ulim}
Fix an ultrafilter $\omega\in \beta X$, and a compatible family $\{t_\alpha:D_\alpha\to R_\alpha\}_{\alpha\in X(\omega)}$. 

Define a function $d_\omega:X(\omega)\times X(\omega)\to [0,\infty)$ by the formula.
$$
d_\omega(\alpha,\beta)=\lim_{x\to \omega} d(t_\alpha(x),t_\beta(x)).
$$

Then $d_\omega$ is a metric on $X(\omega)$ that does not depend on the choice of the compatible family $\{t_\alpha\}$.  Moreover, 
$$
\{d_\omega(\alpha,\beta)~|~\alpha,\beta \in X(\omega)\}\subseteq \{d(x,y)~|~x,y\in X\}
$$
and 
$$
\max_{\alpha\in X(\omega)}|B_{X(\omega)}(\alpha;r)|\leq \max_{x\in X}|B_X(x;r)|,
$$
whence in particular the metric space $(X(\omega),d_\omega)$ is strongly discrete and of bounded geometry.  
\end{proposition}

\begin{proof}
With notation as in the statement, note first that $\omega(D_\alpha\cap D_\beta)=1$ and $\sup_{x\in D_\alpha\cap D_\beta}d(t_\alpha(x),t_\beta(y))<\infty$, whence the limit defining $d_\omega$ makes sense.  We will first show that $d_\omega$ does not depend on the family $\{t_\alpha\}_{\alpha\in X(\omega)}$ of partial translations.  As we clearly have $d_\omega(\alpha,\beta)=d_\omega(\beta,\alpha)$ for any $\alpha,\beta\in X(\omega)$, it suffices to show that for each fixed $\beta$, if we replace $t_\alpha:D_\alpha\to R_\alpha$ with some $s:D_{s}\to R_s$ such that $s(\omega)=\alpha$, then 
$$
\lim_{x\to \omega} d(t_\alpha(x),t_\beta(x))=\lim_{x\to\omega} d(s(x),t_\beta(x)).
$$
Lemma \ref{esssame} implies that if 
$$
D=\{x\in D_\alpha\cap D_s~|~t_\alpha(x)=s(x)\},
$$
then $\omega(D)=1$.  Note that the limits 
$$
\lim_{x\to \omega} d(t_\alpha(x),t_\beta(x)),~~~\lim_{x\to \omega} d(s(x),t_\beta(x))
$$
are unaffected if we only use the restrictions of the functions $x\mapsto d(t_\alpha(x),t_\beta(x))$ and $x\mapsto d(s(x),t_\beta(x))$ to $D\cap D_\beta$, whence they are the same as required.  

We now claim that we may assume that for any fixed $\alpha,\beta\in X(\omega)$, the function $x\mapsto d(t_\alpha(x),t_\beta(x))$ is constant.  Indeed, as $t_\alpha$ and $t_\beta$ are partial translations and $X$ is strongly discrete, this function can only take finitely many distinct values, say $r_1,...,r_k$.  For $i\in \{1,...,k\}$, define
$$
D_i:=\{x\in D_\alpha\cap D_\beta~|~d(t_\alpha(x),t_\beta(x))=r_i\}.
$$
Then there must exist precisely one $i\in \{1,...,k\}$ such that $\omega(D_i)=1$; replacing $t_\alpha$ and $t_\beta$ with their restrictions to this $D_i$ establishes the claim.  

Given this claim, the remaining parts of the statement follow easily on comparing the values of $d_\omega$ on $X(\omega)\times X(\omega)$ to those of $d$ on $X\times X$.
\end{proof}

\begin{definition}\label{limspace}
For each non-principal ultrafilter $\omega$ on $X$, the metric space $(X(\omega),d_\omega)$ is called the \emph{limit space} of $X$ at $\omega$.  
\end{definition}

\begin{proposition}\label{prop:Xomega-same}
Fix an ultrafilter $\omega\in\beta X$. For any $\alpha\in X(\omega)$ we have $X(\alpha)=X(\omega)$ as metric spaces.
\end{proposition}

\begin{proof}
Remark \ref{rem:comp-alt} shows that compatibility is an equivalence relation, which implies $X(\alpha)=X(\omega)$ as sets.  We now show that the metrics $d_\alpha$ and $d_\omega$ are in fact the same: for $\beta,\gamma\in X(\omega)$, we have
$$
d_\alpha(\beta,\gamma) = \lim_{x\to\alpha}d(t_\beta\circ t_\alpha^{-1}(x),t_\gamma\circ t_\alpha^{-1}(x)).
$$
Applying Remark \ref{rem:comp-alt} again, we see that the above limit is in fact equal to
$$
\lim_{x\to\omega}d(t_\beta(x),t_\gamma(x))=d_\omega(\beta,\gamma),
$$
which finishes the proof.
\end{proof}

It is perhaps not obvious at this point what aspect of the geometry of $X$ a limit space $X(\omega)$ is capturing.  We will spend the rest of this section trying to make this a bit clearer in a way that will be useful later: the following proposition makes limit spaces a bit more concrete, and allows us to give some examples.

\begin{proposition}\label{ultequiv}
Let $\omega$ be a non-principal ultrafilter on $X$, and $\{t_\alpha:D_\alpha \to R_\alpha\}$ a compatible family for $\omega$.

For each finite subset $F$ of $X(\omega)$ there exists a subset $Y$ of $X$ with $\omega(Y)=1$, and such that for each $y\in Y$ there is a finite subset $G(y)$ of $X$ such that the map
$$
f_y:F\to G(y),~~~\alpha\mapsto t_\alpha(y)
$$ 
is a surjective isometry. 
\end{proposition}

Thus in some sense, the geometry of $X(\omega)$ models the geometry of $X$ `around' sets of $\omega$-measure one.

\begin{proof}[Proof of Proposition \ref{ultequiv}]
As the metric on $X$ is strongly discrete and using the definition of $d_\omega$, we must have that for each $\alpha,\beta\in X(\omega)$, the set
$$
Y_{\alpha\beta}:=\{x\in D_\alpha~|~d(t_\alpha(x),t_\beta(x))=d_\omega(\alpha,\beta)\}
$$
has $\omega$-measure one.  Set 
$$
Y=\bigcap_{\alpha,\beta\in F}Y_{\alpha\beta};
$$
as this is a finite intersection of subsets of $X$ of $\omega$-measure one, it too has $\omega$-measure one.  For each $y\in Y$, define 
$$
G(y):=\{t_\alpha(y)~|~\alpha\in F\}
$$
and define $f_y:F\to G(y)$ by $f_y(\alpha)=t_\alpha(y)$; these sets and maps have the desired properties.  
\end{proof}

\begin{definition}\label{def:loco}
Let $\omega$ be a non-principal ultrafilter on $X$, $\{t_\alpha\}$ a compatible family for $\omega$, and $F$ a finite subset of $X(\omega)$.  We call a collection $\{f_y:F\to G(y)\}_{y\in Y}$ with the properties in Proposition \ref{ultequiv} above a \emph{local coordinate system} for $F$ (with respect to $\{t_\alpha\}$).  The maps $f_y:F\to G(y)$ are called \emph{local coordinates}.
\end{definition}

\begin{remark}\label{rem:local-coords-balls}
Assume that $F\subseteq X(\omega)$ is a metric ball $B(\omega;r)$.  Then for any compatible family $\{t_\alpha\}$, there exists a local coordinate system $\{f_y:F\to G(y)\}$ for $F$ with respect to $\{t_\alpha\}$ such that each $G(y)$ is the ball $B(y;r)$.  Indeed, take any local coordinate system $\{f_y:F\to G(y)\}_{y\in Y_0}$.  It is clear that the range of any local coordinate $f_y:B(\omega;r)\to G(y)$ is a \emph{subset} of $B(y;r)$, and thus we may define $Y=\{y\in Y_0\mid f_y\text{ is onto } B(y;r)\}$.

It suffices to show that $\omega(Y)=1$.  Note that for every $y\in Y_0\setminus Y$, there exists $x_y\in B(y;r)\setminus f_y(B(\omega;r))$.  Define $t: Y_0\setminus Y\to X$ by $t(y)=x_y$.  If $\omega(Y)=0$, then $t$ is compatible with $\omega$ and necessarily $t(\omega)\in B(\omega;r)$.  This implies that $x_y$ is in the range of $f_y$ for all $y\in Y_0\setminus Y$, which is a contradiction.
\end{remark}

\begin{remark}\label{ghlimits}
Using some language from metric geometry, the spaces $X(\omega)$ admit the following alternative description; we will not use this in what follows, but thought it might be useful to some readers to point it out. 

Given $\omega\in\beta X$, one can use the above proposition to find a sequence $(x_n)$ in $X$ that tends to infinity and such that the pointed metric space $(X(\omega),\omega)$ is the pointed Gromov-Hausdorff limit of the sequence of pointed metric spaces $(X,x_n)$.  Conversely, any limit space of $X$ arises as a pointed Gromov-Hausdorff limit in this way. However, Gromov-Hausdorff limits traditionally concern \emph{isometry classes} of spaces, not the spaces themselves. For the purposes of this paper, it is important that the spaces $(X,x_n)$ converge to $(X(\omega),\omega)$ in a \emph{specific} way, as seen in the previous proposition.
\end{remark}

\begin{examples}\label{limspaceex}
In the following examples, we look at which metric spaces can arise as limit spaces of a given metric space.  We leave the justifications - which are not difficult, given Proposition \ref{ultequiv} - to the reader. 
\begin{enumerate}
\item Let $X=G$ be a discrete group, equipped with any (strongly discrete, bounded geometry) metric that is invariant under the natural left action of $G$ on itself.  Then all limit spaces of $G$ are isometric to $G$ with the given metric (see Lemma \ref{limspgp}).

In particular, if $X=\Z^N$ equipped with any metric defined by restricting a norm from $\R^N$, then all limit spaces are isometric to $\Z^N$ (with the same metric).
\item If $X=\N$ with its usual metric, then all limit spaces are isometric to $\Z$ with its usual metric.  
\item If $X=\{(x,y)\in \R^2~|~x,y\in \N\}$ with the subspace metric, then all limit spaces are isometric to one of
\begin{align*}
 \{(x,y)\in \R^2~|~x\in& \N,y\in \Z\} ,~~~\{(x,y)\in \R^2~|~x\in \Z,y\in \N\}, \\
 & \{(x,y)\in \R^2~|~x,y\in \Z\},
\end{align*}
corresponding to ultrafilters that contain a vertical ray, a horizontal ray, and neither, respectively.
Of course, the first and second of these are themselves isometric! - nonetheless, we thought it would be useful to list them separately as they arise naturally in these forms.
\item If $\|\cdot\|_\infty$ denotes the $\ell^\infty$ norm on $\R^2$ and 
$$
X=\{(x,y)\in \R^2~|~x,y\in \Z,\|(x,y)\|_\infty=n^2 \text{ for some }n\in \N\}
$$
with the restricted $\ell^\infty$ metric, then all limit spaces of $X$ are isometric to $\Z$.
\item Say $G$ is a discrete group generated by a finite set $S$.  Define 
$$
S^{\pm n}=\{g\in G~|~g=s_1^{\pm1}\cdots s_m^{\pm 1},~m\leq n,~s_i\in S\},
$$
and define the \emph{word metric} on $G$ by 
$$
d(g,h)=\min\{n~|~g^{-1}h\in S^{\pm n}\}.
$$
Let
$$
G=G_0\unrhd G_1\unrhd G_2\unrhd\cdots
$$
be a nested sequence of finite index normal subgroups of $G$ such that $\cap G_n=\{e\}$.  The \emph{box space} associated to this data is the disjoint union $X=\sqcup_n (G/G_n)$, where each finite group $G/G_n$ is equipped with the word metric associated to the (image of the) fixed finite generating set of $G$, and $X$ is equipped with any metric that restricts to these metrics on $G/G_n$ and satisfies 
$$
d(G_n,G_n)\to \infty \text{ as } n,m\to\infty~, n\neq m.
$$
Examples of this form have been intensively studied in coarse geometry: see for example \cite[Sections 11.3 and 11.5]{Roe:2003rw}, \cite{Arzhantseva:2011vn}, and \cite{Oyono-Oyono:2009ua}.  All limit spaces of a box space are isometric to $G$, equipped with the given word metric.  Note that the previous example can be identified with the special case of this one where $G=\Z$, the generating set is $\{1\}$, and the subgroups are $G_n=8n^2\Z$.
\end{enumerate}
\end{examples}

\section{Limit operators}\label{sec:limitops}

Throughout the section $X$ denotes a space as in Definition \ref{sdbg}, $E$ denotes a fixed Banach space, and $p$ is a fixed number in $(1,\infty)$.

We will consider `limits at infinity' of band-dominated operators; the following definition formalises the requirement that such limits exist.

\begin{definition}\label{richdef}
Let $\omega$ be a non-principal ultrafilter on $X$.  An operator $A$ in $\AX$ is \emph{rich at $\omega$} if for any pair of partial translations $t,s$ compatible with $\omega$, the limit
$$
\lim_{x\to\omega}A_{t(x)s(x)}
$$
exists for the norm topology on $\mathcal{L}(E)$.  We denote $\AXro$ the collection of band-dominated operators that are rich at $\omega$.

If $A$ is rich at $\omega$ for all $\omega$ in $\partial X$, it is said to be \emph{rich}. We denote $\AXr$ the collection of all rich band-dominated operators.
\end{definition}

\begin{remark}\label{richrem}
Observe that if $E$ is finite dimensional, then all the operators in $\AX$ are automatically rich.  Indeed, this follows as all the matrix entries of some band-dominated $A$ are contained in the ball about zero of radius $\|A\|$ in $E$, which is compact.
\end{remark}

\begin{remark}\label{rem:rich-band-approx}
The reader may wonder whether rich \emph{band-dominated} operators are automatically approximable in norm by rich \emph{band} operators. We show that this is indeed the case in Theorem \ref{denserich}. However, we require an extra assumption on the space $X$ (namely Yu's property A, introduced in Section \ref{sec:parametrices}), so we postpone proving and using this result until Section \ref{sec:parametrices}, where the technique naturally fits.
\end{remark}

\begin{definition}\label{limopdef}
Let $\omega$ be a non-principal ultrafilter on $X$, and $A$ be a band-dominated operator on $\ell^p_E(X)$ that is rich at $\omega$.   Fix a compatible family $\{t_\alpha\}_{\alpha\in X(\omega)}$ for $\omega$.  

The \emph{limit operator of $A$ at $\omega$}, denoted $\Phi_\omega(A)$, is the $X(\omega)$-by-$X(\omega)$ indexed matrix with entries in $\mathcal{L}(E)$ defined by
$$
A_{\alpha\beta}:=\lim_{x\to\omega}A_{t_\alpha(x)t_\beta(x)}.
$$
\end{definition}

The following lemma is immediate from Lemma \ref{esssame}.

\begin{lemma}\label{limopwd}
Let $\omega$ be a non-principal ultrafilter on $X$, and $A$ be a band-dominated operator on $\ell^p_E(X)$ that is rich at $\omega$.   The limit operator $\Phi_\omega(A)$ does not depend on the choice of compatible family. \qed
\end{lemma}

We emphasise at this point that the limit operator $\Phi_\omega(A)$ is a fairly formal object: it is only an abstractly defined matrix, and in particular does not obviously operate on anything!  The next proposition, which is a development of Proposition \ref{ultequiv}, will help us to make limit operators a little more concrete.

\begin{proposition}\label{concretelimop}
Let $\omega$ be a non-principal ultrafilter on $X$, and $A$ be a band-dominated operator on $\ell^p_E(X)$ that is rich at $\omega$.  Let $F$ be a finite subset of $X(\omega)$ and $\epsilon>0$.  Let $\{t_\alpha\}$ be a compatible family of partial translations for $\omega$.

Then there exists a local coordinate system $\{f_y:F\to G(y)\}_{y\in Y}$ as in Definition \ref{def:loco} such that for each $y\in Y$, if
$$
U:\ell^p_E(F)\to \ell^p_E(G(y)),~~~(U\xi)(x):=\xi(f_y^{-1}(x))
$$
is the linear isometry induced by $f_y$, then (recalling the notation for idempotents from Definition \ref{Xcompact}), we have that
$$
\|U^{-1}P_{G(y)}AP_{G(y)}U-P_F\Phi_\omega(A)P_F\|<\epsilon,
$$
where we think of $P_F\Phi_\omega(A)P_F$ as a finite $F$-by-$F$ matrix, with entries in $\mathcal{L}(E)$, acting on $\ell^p_E(F)$ by matrix multiplication.
\end{proposition}

\begin{proof}
For each $\alpha,\beta\in X(\omega)$ set
$$
V_{\alpha\beta}=\{x\in D_\alpha\cap D_\beta~|~\|A_{t_\alpha(x)t_\beta(y)}-\Phi_\omega(A)_{\alpha\beta}\|<\epsilon/|F|^2\}.
$$
Then by definition of $d_\omega$ and $\Phi_\omega(A)$, the sets $V_{\alpha\beta}$ have $\omega$-measure one, whence $Y_{\alpha\beta}:=V_{\alpha\beta}\cap Y_0$ has $\omega$-measure one too.  Now set 
$$
Y=\bigcap_{\alpha,\beta\in F}Y_{\alpha\beta},
$$
which again has $\omega$-measure one as it is a finite intersection of sets with $\omega$-measure one.  The choices of $V_{\alpha\beta}$ guarantee that this $Y$ has the right properties.  
\end{proof}

\begin{corollary}\label{limopbound}
Let $\omega$ be a non-principal ultrafilter on $X$, and $A$ be a band-dominated operator on $\ell^p_E(X)$ that is rich at $\omega$.   Then the matrix $\Phi_\omega(A)$ defines a bounded operator on $\ell^p_E(X(\omega))$ with norm at most $\|A\|$.
\end{corollary}

\begin{proof}
Assume for contradiction this is not the case.  Then there exists a finite subset $F$ of $X(\omega)$ such that $\|P_F\Phi_\omega(A)P_F\|$ has norm at least $\|A\|+\epsilon$ for some $\epsilon>0$.  Using Proposition \ref{concretelimop}, however, it follows that there exists a finite subset $G$ of $X$ which is isometric to $F$ and such that
$$
|\|P_{G}AP_{G}\|-\|P_F\Phi_\omega(A) P_F\||<\epsilon/2.
$$
As $\|P_{G}AP_{G}\|\leq \|A\|$, this is a contradiction.
\end{proof}

Corollary \ref{limopbound} allows us to canonically identify each limit operator $\Phi_\omega(A)$ with a concrete bounded operator on $\ell^p_E(X(\omega))$; we will do this without further comment from now on.  

The next lemma is the final technical ingredient we need to derive the main properties of limit operators.  

\begin{lemma}\label{prodmatent}
Let $\omega$ be a non-principal ultrafilter on $X$, let $A$ and $B$ be operators in $\AX$ that are rich at $\omega$.  Let $\{t_\alpha\}$ be a compatible family for $\omega$.  Then for any $\alpha$, $\beta$ in $X(\omega)$, the limit 
$$
\lim_{x\to\omega}(AB)_{t_\alpha(x)t_\beta(x)}
$$
exists for the norm topology on $\mathcal{L}(E)$ and equals $(\Phi_\omega(A)\Phi_\omega(B))_{\alpha\beta}$.
\end{lemma}

The main point of the proof is the observation that if $A$ and $B$ are band operators, then then the number of terms in the sum expressing any matrix entry in the product $AB$ is uniformly bounded. Hence, the norm convergence of the appropriate entries of $A$ and $B$ to the appropriate entries of $\Phi_\omega(A)$ and $\Phi_\omega(B)$ will imply the same for the products. Making this idea precise and working for band-dominated operators requires making approximations along the way.

\begin{proof}
Fix $\alpha,\beta\in X(\omega)$.  Note that the operator $(\Phi_\omega(A)\Phi_\omega(B))_{\alpha\beta}\in \mathcal{L}(E)$ makes sense by the assumption that $A$ and $B$ are rich at $\omega$.  It suffices to show that for each $\epsilon>0$ 
$$
\omega\left(\{x\in X~|~\|(AB)_{t_\alpha(x)t_\beta(x)}-(\Phi_\omega(A)\Phi_\omega(B))_{\alpha\beta}\|<\epsilon\}\right)=1.
$$
As $A$ and $B$ are band-dominated, there exists $r>0$ such that if $x$ is any point in $X$ and $G$ is a subset of $X$ containing the ball $B(x;r)$, then
\begin{equation}\label{comp2}
\|P_{\{x\}}A(1-P_G)\|<\frac{\epsilon}{3\|B\|}.
\end{equation}
On the other hand, Corollary \ref{limopbound} implies that $\Phi_\omega(A)$ and $\Phi_\omega(B)$ are bounded, whence there exists a finite subset $F$ of $X(\omega)$ such that 
\begin{equation}\label{complim}
\|P_{\{\alpha\}}\Phi_\omega(A)(1-P_F)\|<\frac{\epsilon}{3\|B\|}.
\end{equation}
Expanding $F$ if necessary, we may assume that $F$ contains both the balls $B(\alpha;r)$ and $B(\beta;r)$.  Note that line \eqref{complim} implies that 
\begin{align}\label{ep/31}
\| (\Phi_\omega(A)& \Phi_\omega(B))_{\alpha\beta} -(\Phi_\omega(A)P_F\Phi_\omega(B))_{\alpha\beta}\|  \nonumber\\
& = \|P_{\{\alpha\}}\Phi_\omega(A)(1-P_F)\Phi_\omega(B)P_{\{\beta\}} \|  \nonumber \\ 
& \leq \|P_{\{\alpha\}}\Phi_\omega(A)(1-P_F)\|\|\Phi_\omega(B)P_{\{\beta\}}\| < \epsilon/3
\end{align}
using Corollary \ref{limopbound}.  

Let $\{f_y:F\to G(y)\}_{y\in Y_0}$ be a local coordinate system for $F$ as in Definition \ref{def:loco}.  The assumptions on $r$ as in line \eqref{comp2} and the fact that $f_y$ is an isometry imply that for each $y\in Y_0$,
\begin{align}\label{ep/32}
\| & (AB)_{t_\alpha(y)t_\beta(y)} -(AP_{G(y)}B)_{t_\alpha(y)t_\beta(y)}\| \nonumber \\ 
& \leq \|P_{\{t_\alpha(y)\}}A(1-P_{G(y)})\|\|BP_{\{t_\beta(y)\}}\| < \epsilon/3.
\end{align}

On the other hand, for each $\gamma\in F$, set 
$$
Y_{\alpha\gamma}:=\{y\in Y_0\mid \|A_{t_\alpha(y)t_\gamma(y)}-\Phi_\omega(A)_{\alpha\gamma}\|<\frac{\epsilon}{3|F|\|B\|}\}
$$
and similarly 
$$
Y_{\gamma\beta}:=\{y\in Y_0\mid  \|B_{t_\gamma(y)t_\beta(y)}-\Phi_\omega(B)_{\gamma\alpha}\|<\frac{\epsilon}{3|F|\|A\|}\}.
$$
The definition of $\Phi_\omega(A)$ implies that $\omega(Y_{\alpha\gamma})=1$, and similarly $\omega(Y_{\gamma\beta})=1$, whence as $F$ is finite, if we define $Y:=\cap_{\gamma\in F} (Y_{\alpha\gamma}\cap Y_{\gamma\beta})$, then $\omega(Y)=1$.   We have that for every $y\in Y$,
\begin{align*}
\|(AP_{G(y)}B&)_{t_\alpha(y)t_\beta(y)}-(\Phi_\omega(A)P_F\Phi_\omega(B))_{\alpha\beta}\| \\ &\leq \sum_{\gamma\in F}\|A_{t_\alpha(y)t_\gamma(y)}B_{t_\gamma(y)t_\beta(y)}-\Phi_\omega(A)_{\alpha\gamma}\Phi_\omega(B)_{\gamma\beta}\|<
\epsilon/3.
\end{align*}
Combining this with lines \eqref{ep/31} and \eqref{ep/32} gives
\begin{align*}
\| &  (AB)_{t_\alpha(y)t_\beta(y)}-(\Phi_\omega(A) \Phi_\omega(B))_{\alpha\beta}\| \\
  &\leq \|(AB)_{t_\alpha(y)t_\beta(y)} -(AP_{G(y)}B)_{t_\alpha(y)t_\beta(y)}\| \\
  &~~~+\|(AP_{G(y)}B)_{t_\alpha(y)t_\beta(y)}-(\Phi_\omega(A)P_F\Phi_\omega(B))_{\alpha\beta}\|\\
  &~~~+\|(\Phi_\omega(A)P_F\Phi_\omega(B))_{\alpha\beta} - (\Phi_\omega(A) \Phi_\omega(B))_{\alpha\beta}\|<\epsilon
\end{align*}
for all $y\in Y$ as required.
\end{proof}

\begin{corollary}\label{richalg}
$\AXro$ is a closed subalgebra of $\AX$.
\end{corollary}

\begin{proof}
Note first that if $(A_n)$ is a sequence in $\AX$ that converges in norm to $A$, then the matrix entries of each $A_n$ converge uniformly to those of $A$, i.e.\ 
$$
\sup_{x,y\in X}\|(A_n)_{xy}-A_{xy}\|\to 0 \text{ as } n\to \infty.
$$
It follows from this that if each $A_n$ is rich at $\omega$, then for any $\alpha,\beta\in X(\omega)$ and corresponding partial translations $t_\alpha$, $t_\beta$, we have
$$
\lim_{x\to \omega} A_{t_\alpha(x)t_\beta(x)}=\lim_{n\to\infty}\lim_{x\to\omega}(A_n)_{t_\alpha(x)t_\beta(x)},
$$
and in particular the limit on the left exists.  Hence $A$ is rich at $\omega$, and thus the collection of operators that are rich at $\omega$ is closed.  

It is easy to see that $\AX^{\$,\omega}$ is closed under scalar multiplication and addition.  Lemma \ref{prodmatent} implies that it is closed under multiplication, so we are done.
\end{proof}

The following theorem collects together some important properties of the process of taking a limit operator.

\begin{theorem}\label{richhom}
Let $\omega$ be a non-principal ultrafilter on $X$, and recall that $\AXro$ denotes the Banach algebra of band-dominated operators on $\ell^p_E(X)$ that are rich at $\omega$.  Then the map
$$
\Phi_\omega:\AXro\to \mathcal{L}(\ell^p_E(X(\omega)))
$$
that takes each element of the left-hand-side to its limit operator at $\omega$ has the following properties.   
\begin{enumerate}[(1)]
\item \label{cont} $\Phi_\omega$ is contractive: for all $A\in \AXro$, we have
$$
\|\Phi_\omega(A)\|_{\mathcal{L}(\ell^p_E(X(\omega)))}\leq \|A\|_{\mathcal{L}(\ell^p_E(X))}.
$$
\item \label{prop} $\Phi_\omega$ takes band operators to band operators, and does not increase propagation.
\item \label{homo} $\Phi_\omega$ is a homomorphism.
\end{enumerate}
\end{theorem}

\begin{proof}
Point \eqref{cont} is just Corollary \ref{limopbound}.  

For point \eqref{prop}, assume that $\Phi_\omega(A)_{\alpha\beta}\neq0$ where $\alpha,\beta\in X(\omega)$ satisfy $d_\omega(\alpha,\beta)=r$.  It suffices to show that there exist $x,y\in X$ such that $d(x,y)=r$ and $A_{xy}\neq 0$.  Say $\|\Phi_\omega(A)_{\alpha\beta}\|=\epsilon>0$.  Then for any partial translations $t_\alpha,t_\beta$ taking $\omega$ to $\alpha$, $\beta$ respectively, we have that if 
$$
Y:=\{x\in X~|~\|A_{t_\alpha(x)t_\beta(x)}-\Phi_\omega(A)_{\alpha\beta}\|\}<\epsilon/2\},
$$
then $\omega(Y)=1$.  Passing to a subset of $Y$ of $\omega$-measure one, we may assume that $d(t_\alpha(x),t_\beta(x))=r$ for all $x\in Y$.  In particular, for any $x\in Y$, $\|A_{t_\alpha(x)t_\beta(x)}\|\geq \epsilon/2>0$ and $d(t_\alpha(x),t_\beta(x))=r$, which forces the propagation of $A$ to be be at least $r$ as required.  

Point \eqref{homo} follows by a check of matrix coefficients: for linearity this is clear, while Lemma \ref{prodmatent} says exactly that for any $\alpha,\beta\in X(\omega)$
$$
\Phi_\omega(AB)_{\alpha\beta}=(\Phi_\omega(A)\Phi_\omega(B))_{\alpha\beta},
$$
and thus multiplication is also preserved.
\end{proof}

\begin{remark}\label{*hom}
Note that if $E$ is a Hilbert space (and $p=2$), then $\AXtwoE$ is a $C^*$-algebra, and each $\AXrotwo$ is a $C^*$-subalgebra (i.e.\ is closed under taking adjoints).  Moreover, the homomorphisms $\Phi_\omega$ are $*$-homomorphisms.
\end{remark}

\begin{definition}\label{opspec}
Let $A\in \AXr$ be a rich band-dominated operator.  The collection 
$$
\opspec(A):=\{\Phi_\omega(A)\in \mathcal{L}(\ell^p_E(X(\omega)))~|~\omega\in \partial X\}
$$
is called the \emph{operator spectrum} of $X$.
\end{definition}

In the next three sections, we will discuss how the operator spectrum can be used to detect Fredholmness.

We conclude this section with some simple examples.  We leave the justifications, which are not difficult, given Proposition \ref{concretelimop} and Examples \ref{limspaceex}, to the reader.

\begin{examples}\label{limopex}
\begin{enumerate}
\item Let $X=\N$, so all limit spaces of $\N$ identify canonically with $\Z$ as in Example \ref{limspaceex}, part (2).  Let $V$ be the unilateral shift operator on $\ell^p(\N)$.  Then all limit operators identify with the bilateral shift on $\ell^p(\Z)$.  
\item Similarly, if $X=\N$, consider $\ell^2(\N)$ as identified with the Hardy space $H^2$ of the disk in the usual way.  Let $T_f$ be a Toeplitz operator on $H^2$ with continuous symbol $f:S^1\to \C$.  Then all limit operators of $T$ correspond to the symbol $f$, considered as acting by convolution on $\ell^2(\Z)$ via its Fourier transform.
\item Let $X$ be a general space, and $f\in \ell^\infty(X)$ act on $\ell^p(X)$ by multiplication.  Then all limit operators are multiplication operators on $\ell^p(X(\omega))$.  In general, the collection could be very complicated.  However, in some cases it simplifies substantially, even if one cannot compute what the limit spaces $X(\omega)$ themselves are.  For example, assume $f$ is \emph{slowly oscillating} in the sense that for all $r,\epsilon>0$ there exists a finite subset $F$ of $X$ such that 
$$
\sup_{x\in X\setminus F}\sup_{d(x,y)\leq r}|f(x)-f(y)|<\epsilon
$$
(compare \cite[Section 2.4]{Rabinovich:2004qy} which looks at this class when $X=\Z^N$, and also \cite[Chapter 5]{Roe:1993lq}, which discusses such functions from the point of view of the \emph{Higson corona} and applications to manifold topology).
Then all limit operators are scalars.  The operator spectrum, as a set of scalars, identifies with the set 
$$
\bigcap_{\text{finite}\,F\subseteq X} \overline{f(X\setminus F)}\subseteq \C.
$$
\end{enumerate}
\end{examples}

\section{The main theorem}\label{sec:mainthm}

We are now ready to state the main theorem of this paper, which characterises when a band-dominated operator is \pfred~in terms of limit operators.  This theorem does not hold without further assumptions on the underlying space: we need to assume \emph{property A} in the sense of Yu \cite[Section 2]{Yu:200ve}.  We will introduce this property in two distinct forms later in the paper at the points it is needed.  Suffice to say for now that many natural examples of metric spaces have property A: for example, many negatively curved spaces \cite{Roe:2005rt}, finite dimensional non-positively curved cube complexes \cite{Brodzki:2009fk}, and all countable subgroups of groups of invertible matrices (over any field) \cite{Guentner:2005xr} have property A.

\begin{theorem}\label{main}
Let $X$ be a space as in Definition \ref{sdbg}, $p\in(1,\infty)$ and let $E$ be a Banach space. Assume that $X$ has property A. Let $A$ be a rich band-dominated operator on $\ell^p_E(X)$. Then the following are equivalent:
\begin{enumerate}[(1)]
\item $A$ is \pfred;
\item all the limit operators $\Phi_\omega(A)$ are invertible, and $\sup_{\omega\in \partial X}\|\Phi_\omega(A)^{-1}\|$ is finite;
\item all the limit operators $\Phi_\omega(A)$ are invertible.
\end{enumerate}
\end{theorem}

\begin{remark}\label{parabdo}
Note that the definition of $A$ being a \pfred\ operator (Definition \ref{Xfredholm} above) requires the existence of a bounded operator $B$ on $\ell^p_E(X)$ that is an inverse for $A$ modulo $\KX$.  Inspection of the proof of Theorem \ref{main} shows that we can do a bit better: $B$ can be taken to be a band-dominated operator, provided $X$ has property A.
\end{remark}

At the end of this section, we give the proof of (1) implies (2), which follows along similar lines to that of \cite[Proposition 1.2.9]{Rabinovich:2004qy}, and does not require the property A assumption.  The implication (2) implies (3) is of course trivial.  

In the next two sections, we prove (3) implies (2) and (2) implies (1) completing the proof of the theorem.  Neither of these results seems to admit an easy proof: in particular, both make non-trivial uses of property A.  They do this in quite different guises, however: (3) implies (2) uses a version of the metric sparsification property of Chen, Tessera, Wang, and Yu \cite{Chen:2008so}, while (2) implies (1) uses the existence of `slowly varying' partitions of unity as introduced by Guentner and Dadarlat \cite{Dadarlat:2007qy}.

We do not know if property A is necessary for the implication (3) implies (2), although we suspect it probably is.  It is certainly necessary for the implication (2) implies (1), as is discussed in Section \ref{sec:ghosts}.

\medskip

Embarking now on the proof of Theorem \ref{main}, (1) implies (2), we separate a part of the proof as an auxiliary Lemma:

\begin{lemma}\label{lem:closevectors}
Let $A$ be band-dominated operator on $\ell^p_E(X)$, rich at $\omega\in\partial X$. For any finitely supported unit vector $v\in \ell^p_E(X(\omega))$, any finite subset $G\subseteq X$ and any $\epsilon>0$, there exists a unit vector $w\in\ell^p_E(X)$, such that 
$$
\left|\|Aw\|-\|\Phi_\omega(A)v\|\right|<\epsilon
$$
and $\supp(w)\cap G=\emptyset$.
\end{lemma}
§
\begin{proof}
Fix $\epsilon>0$ and a finitely supported unit vector $v\in \ell^p_E(X(\omega))$, say supported in some ball $B(\omega,r)\subseteq X(\omega)$. Since $\Phi_\omega(A)$ is a bounded operator on $\ell^p_E(X(\omega))$, there exists $r'\geq r$, such that for any $R\geq r'$
\begin{equation*}
\left\|P_{B(\omega;R)}\Phi_\omega(A)v - \Phi_\omega(A)v\right\| < \epsilon/3.
\end{equation*}
As $A$ is band-dominated, there is a band operator\footnote{We cannot assume that it is rich without assuming property A and appealing to Theorem \ref{denserich}.} $A'$, such that $\|A-A'\|<\epsilon/6$.

Fix now $R=\max\{r',r+2\prop(A')\}$. For any unit vector $w\in \ell^p_E(X)$, supported in a ball $B(x;r)\subseteq X$ for some $x\in X$, we have that
\begin{multline*}
\left\|P_{B(x;R)}Aw - Aw\right\| \leq \left\|P_{B(x;R)}Aw-P_{B(x;R)}A'w\right\| +\\
+  \left\|P_{B(x;R)}A'w-A'w\right\| + \left\|A'w-Aw\right\| < \epsilon/3,
\end{multline*}
since $P_{B(x;R)}A'w=A'w$ as $A'$ can spread the support of $w$ by at most $2\prop(A')$.

Using Proposition \ref{concretelimop}, there is a local coordinate system $\{f_y:B(\omega;R)\to G(y)\}_{y\in Y}$ and corresponding collection of linear isometries $\{U_y:\ell^p_E(B(\omega;R))\to \ell^p_E(G(y))\}_{y\in Y}$, such that 
$$
\|U_y^{-1}P_{G(y)}AP_{G(y)}U_y-P_F\Phi_\omega(A)P_F\|<\epsilon/3.
$$
In particular, for each $y\in Y$ the vector $w_y:=U_yv\in \ell^p_E(X)$ is supported in $B(y;r)$ and so
$$
\left|\|P_{B(y;R)}Aw_y\|-\|P_{B(\omega;R)}\Phi_\omega(A)v\|\right|<\epsilon/3.
$$
Consequently, putting the estimates together,
$$
\left|\|Aw_y\|-\|\Phi_\omega(A)v\|\right|<\epsilon,
$$
for every $y\in Y$. As $Y$ is infinite and $X$ has bounded geometry, we can arrange that $\supp(w_y)\cap G\subseteq B(y;r)\cap G=\emptyset$ for any given finite $G\subseteq X$.
\end{proof}

\begin{proof}[Proof of Theorem \ref{main}, (1) implies (2)]
Let $A$ be a rich band-dominated \pfred\ operator on $\ell^p_E(X)$, so there exists a bounded operator $B$ on $\ell^p_E(X)$ such that $K_1:=AB-1$ and $K_2:=BA-1$ are \pcom~operators.   Let $\omega$ be a non-principal ultrafilter on $X$. We will first show that $\Phi_\omega(A)$ is bounded below independently of $\omega$; more precisely, we will show that $\|\Phi_\omega(A)v\|\geq 1/\|B\|$ for all finitely supported unit vectors $v\in\ell^p_E(X(\omega))$. Fix then some finitely supported $v\in\ell^p_E(X(\omega))$.

Take $\epsilon>0$. Then $\|K_2P_G-K_2\|<\epsilon$ for some finite $G\subseteq X$. Hence any vector $w\in\ell^p_E(X)$ whose support misses $G$ will satisfy $\|K_2w\|<\epsilon$. Now Lemma \ref{lem:closevectors} delivers a unit vector $w\in\ell^p_E(X)$ with 
$$
\left|\|Aw\|-\|\Phi_\omega(A)v\|\right|<\epsilon
$$
and satisfying $\|K_2w\|<\epsilon$. We compute
$$
\|B\|\|Aw\|\geq\|BAw\| = \|(1-K_2)w\|\geq\|w\|-\|K_2w\|\geq 1-\epsilon.
$$
Hence
$$
\|\Phi_\omega(A)v\|\geq \|Aw\|-\epsilon\geq\frac{1-\epsilon}{\|B\|}-\epsilon.
$$
Letting $\epsilon\to 0$ shows that $\|\Phi_\omega(A)v\|\geq1/\|B\|$ as required.

Let now $q$ be the conjugate index of $p$, i.e.\ $q$ satisfies $1/q+1/p=1$.  Note that the adjoints of \pcom~ operators on $\ell^p_E(X)$ are \pcom~on $\ell^q_{E^*}(X)$, and that adjoints of rich band-dominated operators on $\ell^p_E(X)$ are rich band-dominated operators on $\ell^q_{E^*}(X)$.  An analogous argument to the above then shows that $\Phi_\omega(A)^*$ is bounded below by $1/\|B\|$ as an operator on $\ell^p_{E^*}(X(\omega))$, where $q$ is the conjugate index of $p$.  It follows that $\Phi_\omega(A)$ is invertible, and that the norm of its inverse is at most $\|B\|$, as required.
\end{proof}

\section{Partitions of unity, and constructing parametrices}\label{sec:parametrices}

In this section, we prove the implication (2) implies (1) from Theorem \ref{main}.

Throughout this section $X$ is a space as in Definition \ref{sdbg}, $E$ is a Banach space, and $p$ is a fixed number in $(1,\infty)$.  We set $q\in (1,\infty)$ to be the conjugate index of $p$, i.e.\ the unique number such that $1/p+1/q=1$.  

\subsection{Partitions of unity and constructing operators}

\begin{definition}\label{propa}
A \emph{metric $p$-partition of unity} on $X$ is a collection $\{\phi_i:X\to [0,1]\}$ of functions on $X$ satisfying the following conditions.
\begin{enumerate}
\item There exists $N\in \N$ such that for each $x\in X$, at most $N$ of the numbers $\phi_i(x)$ are non-zero.
\item The $\phi_i$ have uniformly bounded supports, i.e.\ 
$$
\sup_i(\diam(\{x\in X~|~\phi_i(x)\not=0\})) < \infty.
$$
\item For each $x\in X$, $\sum_{i\in I} (\phi_i(x))^p=1$.
\end{enumerate}  

Let $r,\epsilon$ be positive numbers.  A metric $p$-partition of unity $\{\phi_i\}$ has \emph{$(r,\epsilon)$-variation} if whenever $x,y\in X$ satisfy $d(x,y)\leq r$, then 
$$
\sum_{i\in I}|\phi_i(x)-\phi_i(y)|^p<\epsilon^p.
$$

The space $X$ has \emph{property A} if for any $r,\epsilon>0$ there exists a metric $p$-partition of unity with $(r,\epsilon)$-variation.  
\end{definition}

\begin{remark}
This definition does not depend on $p$. It is equivalent to the `standard' definition of property A by \cite[Theorem 1.2.4]{Willett:2009rt}. More precisely, the item (6) in this Theorem is precisely the above definition for $p=1$; however the proofs work (with the obvious changes) for any $p\in[1,\infty)$.
\end{remark}

In the rest of this section, we will show how to use partitions of unity to construct a parametrix for an operator satisfying the assumptions of part (2) of Theorem \ref{main}.

\begin{lemma}\label{construct}
Let $\{\phi_i\}_{i\in I}$ be a metric $p$-partition of unity on $X$ as in Definition \ref{propa}.

Let $J$ be a subset of the index set $I$, and assume that we have been given a collection of bounded operators $\{B_i\}_{i\in J}$ on $\ell^p_E(X)$ such that $M:=\sup_i\|B_i\|$ is finite.  Then the sum
$$
\sum_{i\in J}\phi_i^{p/q}B_i\phi_i
$$
converges strongly to a band operator of norm at most $M$ on $\ell^p_E(X)$.
\end{lemma}

\begin{proof}
The conditions on the partition of unity imply that if $v\in \ell^p_E(X)$ has finite support, then only finitely many of the terms in the sum $\sum_{i\in J}\phi_i^{p/q}B_i\phi_iv$ are non-zero. Hence this sum represents a well-defined vector in $\ell^p_E(X)$.  To establish strong convergence, it thus suffices to show that the assignment $v\mapsto \sum_{i\in J}\phi_i^{p/q}B_i\phi_iv$ (defined on the dense subspace of $\lp(X)$ consisting of functions with finite support) is a bounded operator. We show that for any $v\in \ell^p_E(X)$ of finite support, we have that
$$
\Big\|\sum_{i\in J}\phi_i^{p/q}B_i\phi_iv\Big\|\leq M\|v\|
$$
(which of course also establishes the norm bound).  In fact, noting that the dual of $\ell^p_E(X)$ is $\ell^q_{E^*}(X)$, where $q$ is the conjugate index to $p$, it suffices to show that if $v\in \ell^p_E(X)$ and $w\in \ell^q_{E^*}(X)$ have finite support, then 
$$
\Big|\Big\langle \sum_{i\in J}\phi_i^{p/q}B_i\phi_iv,w\Big\rangle \Big|\leq M\|v\|\|w\|,
$$
where $\langle,\rangle$ denotes the canonical pairing between $\ell^p_E(X)$ and $\ell^q_{E^*}(X)$.  

Now, using that the adjoint of a multiplication operator from $l^\infty(X)$ acting on $\ell^p_E(X)$ is the same function acting by multiplication on $\ell^q_{E^*}(X)$, we have 
\begin{align*}
\Big|\Big\langle \sum_{i\in J}\phi_i^{p/q}B_i\phi_iv,w\Big\rangle \Big| & \leq \sum_{i\in J} |\langle B_i\phi_i,\phi_i^{p/q}w\rangle| \\
& \leq \sum_{i\in J} \|B_i\|\|\phi_iv\|\|\phi_i^{p/q}w\| \\
& \leq M\sum_{i\in J} \|\phi_iv\|\|\phi_i^{p/q}w\|.
\end{align*}
H\"{o}lder's inequality applied to the measure space ($J$, counting measure) bounds this above by
\begin{align*}
M\Big(\sum_{i\in J} & \|\phi_iv\|^p\Big)^{1/p} \Big(\sum_{i\in J}\|\phi_i^{p/q}w\|^q\Big)^{1/q} \\& =M\Big(\sum_{i\in J}\sum_{x\in X} (\phi_i(x))^p\|v(x)\|^p\Big)^{1/p}\Big(\sum_{i\in J}\sum_{x\in X}(\phi_i(x))^{p}\|w(x)\|^q\Big)^{1/q} \\
& =M\Big(\sum_{x\in X} \Big(\sum_{i\in J} (\phi_i(x))^p\Big)\|v(x)\|^p\Big)^{1/p}\Big(\sum_{x\in X} \Big(\sum_{i\in J} (\phi_i(x))^p\Big)\|w(x)\|^q\Big)^{1/q} \\
& \leq M\Big(\sum_{x\in X} \|v(x)\|^p\Big)^{1/p}\Big(\sum_{x\in X} \|w(x)\|^q\Big)^{1/q} \\
& =M\|v\|\|w\|,
\end{align*}
which establishes the norm bound.  

The fact that $\sum_{i\in J}\phi_i^{p/q}B_i\phi_i$ is a band operator (with propagation at most $\sup(\diam(\supp(\phi_i)))$) follows directly on looking at matrix coefficients.
\end{proof}

\begin{lemma}\label{almcom}
Let $A$ be a band operator on $\ell^p_E(X)$ with propagation at most $r$, and let $N=\sup_{x\in X}|B(x;r)|$.  Let $\{\phi_i\}_{i\in I}$ be a metric $p$-partition of unity on $X$ with $(r,\epsilon)$-variation in the sense of Definition \ref{propa}.  

Let $J$ be a subset of the index set $I$, and assume that we have been given a collection of bounded operators $\{B_i\}_{i\in J}$ on $\ell^p_E(X)$ such that $M:=\sup_i\|B_i\|$ is finite.  Then the sum
$$
\sum_{i\in J} \phi_i^{p/q}B_i[\phi_i,A]
$$
converges strongly to a band operator of norm at most $\epsilon N\|A\| M$.
\end{lemma}

\begin{proof}
Similarly to the proof of Lemma \ref{construct}, it suffices for the convergence and norm estimate to show that if $v\in \ell^p_E(X)$ and $w\in \ell^q_{E^*}(X)$ have finite support, then 
$$
\Big|\Big\langle \sum_{i\in J} \phi_i^{p/q}B_i[\phi_i,A]v,w\Big\rangle \Big|\leq \epsilon N \|A\|M\|v\|\|w\|.
$$
We may bound the left hand side above by
\begin{align*}
\Big|\Big\langle \sum_{i\in J} \phi_i^{p/q}B_i[\phi_i,A]v,w\Big\rangle \Big| & \leq \sum_{i\in J}|\langle [\phi_i,A]v,B_i^*\phi_i^{p/q}w\rangle| \\
& \leq \sum_{i\in J}\|[\phi_i,A]v\|\|B_i^*\phi_i^{p/q}w\| \\
& \leq M \sum_{i\in J}\|[\phi_i,A]v\|\|\phi_i^{p/q}w\| \\
& \leq M\Big(\sum_{i\in J} \|[\phi_i,A]v\|^p\Big)^{1/p}\Big(\sum_{i\in J}\|\phi_i^{p/q}w\|^q\Big)^{1/q},
\end{align*}
where the last inequality is H\"{o}lder's inequality.  Using the same argument as in Lemma \ref{construct}, the second factor is bounded above by $\|w\|$, so we see that 
\begin{equation}\label{halfway}
\Big|\Big\langle \sum_{i\in J} \phi_i^{p/q}B_i[\phi_i,A]v,w\Big\rangle \Big|\leq M\|w\|\Big(\sum_{i\in J} \|[\phi_i,A]v\|^p\Big)^{1/p}.
\end{equation}

For continuing with the estimates, we decompose $A$ as in Lemma \ref{decomp}, i.e.
\begin{equation}\label{almcomdecomp}
A=\sum_{k=1}^N f_kV_k,
\end{equation}
where each $f_k$ is an operator in $l^\infty(X,\mathcal{L}(E))$ with norm at most $\|A\|$, and each $V_k$ is a partial translation operator on $\ell^p_E(X)$ with propagation at most $r$ (compare Example \ref{bandex}).

We now focus attention on the term $\Big(\sum_{i\in J} \|[\phi_i,A]v\|^p\Big)^{1/p}$.  Fix $i\in J$ for the moment.  Computing, using the sum in line \eqref{almcomdecomp}, that $\phi_i$ commutes with each $f_k$, and that $\|f_k\|\leq \|A\|$ for each $k$ gives
\begin{equation}\label{startcom}
\|[\phi_i,A]v\|^p=\Big\|\sum_{k=1}^N f_k[\phi_i,V_k]v\Big\|^p\leq \|A\|^pN^p \max_{k\in \{1,...,N\}}\|[\phi_i,V_k]v\|^p.
\end{equation}
Write then $V=V_k$ for some fixed $k\in \{1,...,N\}$, and let $t:D\to R$ be the partial translation function underlying $V$ as in Example \ref{bandex}.  Computing for any $x\in X$,
$$ 
\Big([\phi_i,V]v\Big)(x)=\left\{\begin{array}{ll} (\phi_i(t(x))-\phi_i(x))v(t(x)) &x\in D \\ 0 & \text{otherwise}\end{array}\right.
$$
Hence
$$
\|[\phi_i,V]v\|^p=\sum_{x\in D}|\phi_i(t(x))-\phi_i(x)|^p\|v(t(x))\|^p,
$$
and so
\begin{align*}
\sum_{i\in J}\|[\phi_i,V]v\|^p & =\sum_{x\in D}\|v(t(x))\|^p\Big(\sum_{i\in J}|\phi_i(t(x))-\phi_i(x)|^p\Big) \\
& \leq  \sum_{x\in D}\|v(t(x))\|^p \epsilon^p \\
& \leq \|v\|^p\epsilon^p.
\end{align*}
As the choice of $k\in \{1,...,N\}$ was arbitrary, combing this with line \eqref{startcom} gives
$$
\sum_{i\in J}\|[\phi_i,A]v\|^p\leq \|A\|^pN^p\|v\|^p\epsilon^p.
$$
Finally, combing this with line \eqref{halfway} gives the desired norm bound.  

The fact that $\sum_{i\in J} \phi_i^{p/q}B_i[\phi_i,A]$ is a band operator (with propagation at most $\sup(\diam(\supp(\phi_i)))+\prop(A)$) follows directly on looking at matrix coefficients. 
\end{proof}

\subsection{Density of rich band operators}

Our next goal is to show that the rich band operators are dense in the rich band-dominated operators - the analogue of  \cite[Theorem 2.1.18]{Rabinovich:2004qy}.  That result is proved using Fourier analysis, which is not available in our context; instead, we proceed through the following corollary, which is inspired by the Hilbert space case \cite[Lemma 11.17]{Roe:2003rw}.

\begin{corollary}\label{schurish}
Assume $X$ has property A, and for each $n$, let $\{\phi_i^{(n)}\}$ be a metric $p$-partition of unity with $(n,1/n)$-variation.  Define 
$$
M_n:\AX\to \AX,~~~A\mapsto \sum_{i\in I} (\phi^{(n)}_i)^{p/q}A\phi^{(n)}_i.
$$
Then each $M_n$ is a well-defined linear operator of norm one.  Moreover, $M_n(A)\to A$ in norm, as $n\to\infty$, for each $A\in \AX$.
\end{corollary}

\begin{proof}
Lemma \ref{construct} (with $J=I$ and $B_i=A$ for all $i$) implies that $M_n$ is well-defined and norm one.  On the other hand, for each $n$ and any band-operator $A$,
$$
M_n(A)=\sum_{i\in I} (\phi^{(n)}_i)^{p/q}A\phi^{(n)}_i=\sum_{i\in I} (\phi^{(n)}_i)^{1+p/q}A+\sum_{i\in I} (\phi^{(n)}_i)^{p/q}[A,\phi^{(n)}_i].
$$
As $p$ and $q$ are conjugate indices, $1+p/q=p$ and so the first term on the right-hand-side above is $A$ (as $\{\phi_i\}$ is a metric $p$-partition of unity).  On the other hand, Lemma \ref{almcom} (with $J=I$ and $B_i$ the identity for all $i$) implies that the second term on the right-hand-side has norm at most $\|A\|N/n$ for some fixed $N$, and thus tends to zero as $n$ tends to infinity.  It follows that for any band-operator $M_n(A)$ converges in norm to $A$ as $n$ tends to infinity.  The result follows for band-dominated operators as $\|M_n\|\leq 1$ for all $n$.
\end{proof}

\begin{theorem}\label{denserich}
Assume $X$ has property A.  Then the rich band operators on $\ell^p_E(X)$ are dense in the rich band-dominated operators on $\ell^p_E(X)$.  
\end{theorem}

\begin{proof}
Say $A$ is a rich band-dominated operator and $M_n$ is as in Corollary \ref{schurish}; we will first show that $M_n(A)$ is rich.  Let $\{t_\alpha:D_\alpha\to R_\alpha\}$ be a compatible family of partial translations for $\omega$, and let $\alpha,\beta$ be points in $X(\omega)$.  It suffices by completeness of $\mathcal{L}(E)$ to show that for any $\epsilon>0$, there is a set $Y$ of $\omega$-measure one such that for all $x,y\in Y$
\begin{equation}\label{cauchy}
\|M_n(A)_{t_{\alpha}(x)t_\beta(x)}-M_n(A)_{t_\alpha(y)t_\beta(y)}\|<\epsilon.
\end{equation}
Concretely, the matrix coefficients of $M_n(A)_{t_{\alpha}(x)t_\beta(x)}$ are given by
$$
\sum_{i\in I}(\phi^{(n)}_i)^{p/q}(t_\alpha(x))A_{t_\alpha(x)t_\beta(x)}\phi^{(n)}_i(t_\beta(x));
$$
writing
$$
c(x)=\sum_{i\in I}(\phi^{(n)}_i)^{p/q}(t_\alpha(x))\phi^{(n)}_i(t_\beta(x)),
$$
this says that 
$$
M_n(A)_{t_{\alpha}(x)t_\beta(x)}=c(x)A_{t_\alpha(x)t_\beta(x)}.
$$
On the other hand, using that $\phi^{(n)}_i$ is a $p$-partition of unity and H\"{o}lder's inequality, each $c(x)$ is a number in $[0,1]$.  It follows that for some $i\in \{1,...,(\lfloor\frac{\epsilon}{4\|A\|}\rfloor)^{-1}\}$, the set
$$
Z_i:=\left\{x\in D_\alpha\cap D_\beta~|~c(x)\in \left[(i-1)\tfrac{\epsilon}{4\|A\|},i\tfrac{\epsilon}{4\|A\|}\right]\right\}
$$
has $\omega$-measure one.  Set 
$$
Z=\left\{x\in D_\alpha\cap D_\beta~|~\|A_{t_\alpha(x)t_\beta(x)}-\Phi_\omega(A)_{\alpha\beta}\|<\epsilon/4\right\},
$$
which has $\omega$-measure one by richness of $A$.  If we set $Y=Z\cap Z_i$, then $\omega(Y)=1$ and the inequality in line \eqref{cauchy} is satisfied for all $x,y\in Y$.  

The result thus follows from Corollary \ref{schurish}.   
\end{proof}

\subsection{Constructing parametrices}

Most of the rest of this section will be taken up with proving the following slightly technical proposition; as we show below, the part (2) implies (1) from Theorem \ref{main} follows using standard methods from this.

\begin{proposition}\label{main2}
Let $A$ be a rich band operator, and assume that all the limit operators $\Phi_\omega(A)$ are invertible, and 
$$
\sup_{\omega}\|\Phi_\omega(A)^{-1}\|<M
$$
is finite.  Then there exists an operator $B$ in $\AX$ which is an inverse for $A$ modulo $\KX$, and such that $\|B\|\leq 2M$.
\end{proposition}

\begin{proof}[Proof of Theorem \ref{main}, part (2) implies (1)]
Let $A$ be as in the statement of Theorem \ref{main}, and let $(A_n)$ be a sequence of rich band operators that converge to $A$ in norm (Theorem \ref{denserich} implies that such a sequence exists).  Let $N$ be so large that for all $n\geq N$, $\|A_n-A\|<1/(M+1)$.  Using that each $\Phi_\omega$ is a contraction, for $n\geq N$,
$$
\|\Phi_\omega(A_n)-\Phi_\omega(A)\|\leq \|A_n-A\|<\frac{1}{M+1}<\frac{1}{\|\Phi_\omega(A)^{-1}\|}.
$$ 
Hence by a standard Neumann series argument in Banach algebra theory, for each $n\geq N$ and each $\omega$, $\Phi_\omega(A_n)$ is invertible and 
$$
\|\Phi_\omega(A_n)^{-1}\|\leq \frac{1}{1-\|\Phi_\omega(A_n)-\Phi_\omega(A)\|\|\Phi_\omega(A)^{-1}\|}\leq M+1.
$$
Now, as each $A_n$ is a band-operator, Proposition \ref{main2} implies that for each $n\geq N$ there exists $B_n$ which is an inverse for $A_n$ in the Banach algebra $\AX/\KX$, and which satisfies $\|B_n\|\leq 2M+2$ for all $n\geq N$.  Looking at norms in $\AX/\KX$,
$$
\|B_n-B_m\|= \|B_n(A_n-A_m)B_m\|\leq (2M+2)^2\|A_n-A_m\|,
$$
which tends to zero as $n,m$ tend to infinity.  Hence the sequence $(B_n)_{n\geq N}$ tends to some limit $B$ in $\AX/\KX$, which is clearly an inverse for $A$, completing the proof.
\end{proof}
  
For the remainder of this section, we focus on proving Proposition \ref{main2}.  

\begin{lemma}\label{localinverses}
Let $A$ be a rich band operator, satisfying the assumptions of Proposition \ref{main2}.
Let $\{V_i\}_{i\in I}$ be a uniformly bounded, finite multiplicity cover of $X$. Then there exists a finite subset $F$ of $I$ such that for all $i\in I\setminus F$ there is an operator $B_i$ on $\ell^p_E(X)$ of norm at most $M$, and such that
$$
B_iAP_{V_i}=P_{V_i}
$$
for all $i\in I\setminus F$.
\end{lemma}

\begin{proof}
Assume for contradiction that this is not true.  Then there exists a sequence in $I$, say $(i_n)$, that eventually leaves any finite subset of $I$ and is such that if $B$ is a bounded operator on $\ell^p_E(X)$ with $\|B\|\leq M$, then if $Q_n:=P_{V_{i_n}}$ we have
\begin{equation}\label{badass}
BAQ_n\neq Q_n \text{ for all n}.
\end{equation}

Let $s>0$ and $(x_n)$ be such that $V_{i_n}$ is a subset of $B(x_n;s)$ for all $n$.  The assumptions force $(x_n)$ to tend to infinity.  Let $Y=\{x_n~|~n\in \N\}$. As this is an infinite set, there exists a non-principal ultrafilter $\omega$ on $X$ such that $\omega(Y)=1$.  Proposition \ref{concretelimop} then implies that on replacing $(x_n)$ with a subsequence we may assume that we have a local coordinate system: bijections $f_n:B(\omega;s+\prop(A))\to B(x_n;s+\prop(A))$ for each $n$ and associated linear isometries
$$
U_n:\ell^p_E(B(\omega;s+\prop(A))\to \ell^p_E(B(x_n;s+\prop(A))).
$$
As there are only finitely many subsets of $B(\omega;s+\text{prop}(A))$, for some subset $V\subseteq B(\omega;s+\text{prop}(A))$ we have that 
$$
\{n\in \N~|~f_n(V)=V_{i_n}\}
$$
is infinite; thus passing to another subsequence, we may assume that $f_n(V)=V_{i_n}$ for all $n$.  

Write now $Q=P_{V}$ and $P=P_{B(\omega;s+\prop(A))}$, and note that $Q_n=U_nQU_n^{-1}$. Write $P_n=U_nPU_n^{-1}$. Then for any $\epsilon>0$ and all sufficiently large $n$,
$$
\|U_n^{-1}P_nAP_nU_n-P\Phi_\omega(A)P\|<\epsilon
$$
Furthermore, note that $PQ=Q$, and that since $\Phi_\omega(A)$ has propagation at most $\text{prop}(A)$, we also have $\Phi_\omega(A)Q=P\Phi_\omega(A)PQ$ by our choice of $P$.

Then, it follows that all sufficiently large $n$ that
\begin{multline*}
\|\Phi_\omega(A)^{-1}U_n^{-1}P_nAP_nU_nQ-Q\| = \\
= \|\Phi_\omega(A)^{-1}U_n^{-1}P_nAP_nU_nQ-\Phi_\omega(A)^{-1}P\Phi_\omega(A)PQ\| = \\
=\|\Phi_\omega(A)^{-1}(U_n^{-1}P_nAP_nU_n- P\Phi_\omega(A)P)Q\|
 \leq \|\Phi_\omega(A)^{-1}\|\cdot \epsilon.
\end{multline*}
Hence for $\epsilon<1/M$, the operator $1+\Phi_\omega(A)^{-1}U_n^{-1}P_nAP_nU_nQ-Q$ is invertible, and we may define 
$$
B:=(1+\Phi_\omega(A)^{-1}U_n^{-1}P_nAP_nU_nQ-Q)^{-1}\Phi_\omega(A)^{-1}.  
$$
A simple algebraic check shows that 
\begin{equation}\label{algcon}
BU_n^{-1}P_nAP_nU_nQ=Q.
\end{equation}
From a basic computation with Neumann series it follows that if $\|T\|<\delta$, then $\|(1+T)^{-1}-1\|\leq\frac{\delta}{1-\delta}$. Applying this in our situation yields 
\begin{align*}
\|B-\Phi_\omega(A)^{-1}\| & \leq \left\|(1+\Phi_\omega(A)^{-1}U_n^{-1}P_nAP_nU_nQ-Q)^{-1}-1\right\|\|\Phi_\omega(A)^{-1}\| \\
& \leq \frac{\|\Phi_\omega(A)^{-1}\|\cdot\epsilon}{1-\|\Phi_\omega(A)^{-1}\|\cdot\epsilon}\|\Phi_\omega(A)^{-1}\|.
\end{align*}
Hence
$$
\|B\|\leq \|\Phi_\omega(A)^{-1}\|\left(1+\frac{\|\Phi_\omega(A)^{-1}\|\cdot\epsilon}{1-\|\Phi_\omega(A)^{-1}\|\cdot\epsilon}\right),
$$
which is less than $M$ for $\epsilon$ sufficiently small.  Set now 
$$
B_n=U_nBU_n^{-1}P_n,
$$
so $\|B_n\|\leq M$ for all suitably large $n$. Using \eqref{algcon} above and the fact that $P_nQ_n=Q_n$, we get
$$
B_nAQ_n=U_nBU_n^{-1}P_nAP_nU_nQU_n^{-1}=U_nQU_n^{-1}=Q_n.
$$
This contradicts the assumption in line \eqref{badass} at the start of the proof, so we are done.  
\end{proof}

\begin{proof}[Proof of Proposition \ref{main2}]
Let $\epsilon=1/2MN\|A\|$, where $N$ is an upper bound on the number of diagonals of $A$ as in Lemma \ref{decomp}.  Let $\{\phi_i\}_{i\in I}$ be a metric $p$-partition of unity with $(\prop(A),\epsilon)$-variation.  Applying Lemma \ref{localinverses} to the cover $\{\supp(\phi_i)\}$ of $X$ gives a finite subset $F$ of $I$ and operators $B_i$, $i\in I\setminus F$ with the properties in that lemma.  

Note that by Lemma \ref{construct}, the sum $\sum_{i\in I\setminus F}\phi_i^{p/q}B_i\phi_i$ converges strongly to an operator on $\ell^p_E(X)$ of norm at most $M$.  Consider
\begin{align*}
\Big(\sum_{i\in I\setminus F}\phi_i^{p/q}B_i\phi_i\Big)A & =\sum_{i\in I\setminus F}\phi_i^{p/q} B_iAQ_i\phi_i+\underbrace{\sum_{i\in I\setminus F} \phi_iB_i[\phi_i,A]}_{=:T} \\
&=\sum_{i\in I\setminus F} \phi_i^{p/q+1}+T.
\end{align*}
Noting that as $p$ and $q$ are conjugate indices, $p/q+1=p$; as moreover $\sum_i \phi_i^p=1$, this is equal to 
$$
\sum_{i\in I\setminus F} \phi_i^p +T=(1+T)-\sum_{i\in F}\phi_i^p.
$$
Lemma \ref{almcom} implies that $\|T\|\leq 1/2$, whence $1+T$ is invertible and its inverse has norm at most $2$; moreover its inverse is given by a Neumann series and is thus band-dominated as $T$ is a band-operator.  Note moreover that $\sum_{i\in F}\phi_i^p$ is \pcom.  It follows that 
$$
A_L:=(1+T)^{-1}\Big(\sum_{i\in I\setminus F}\phi_i^{p/q}B_i\phi_i\Big)
$$
is a band-dominated operator that is a left inverse for $A$ modulo the \pcom~ operators of norm at most $2M$.  A precisely analogous argument shows that $A$ also has a right inverse, say $A_{R}$, modulo the \pcom~ operators, with $A_R$ also having norm at most $2M$.  It follows that $A$ is invertible in $\AX/\KX$, with it inverse in this algebra equal to the image of $A_L$ and to that of $A_R$.  In particular, $A_L$ and $A_R$ are equal modulo $\KX$.  Thus either of them has the properties required of the operator $B$ in the statement, and we are done.
\end{proof}

\section{Metric sparsification and uniform boundedness}\label{sec:unifbdd}

The goal of this section is to prove the final implication from Theorem \ref{main}: (3) implies (2). The basic strategy is a generalisation of the approach taken by Lindner and Seidel \cite{Lindner:2014aa}.

Before we formulate the main result of this section, we need a definition.
\begin{definition}
Let $E_1,E_2$ be Banach spaces and $T:E_1\to E_2$ be a bounded linear operator.  We define the \emph{lower norm} of $T$ to be
$$
\nu(T): = \inf\left\{\tfrac{\|Tv\|_{E_2}}{\|v\|_{E_1}} \mid v\in E_1\setminus\{0\} \right\}.
$$
If $E_1=E_2=\lp(X)$ and $s\geq 0$, we shall also denote the lower norm computed on vectors supported on a set of diameter at most $s$ by
$$
\nu_s(T): = \inf\left\{\tfrac{\|Tv\|}{\|v\|} \mid v\in \lp(X)\setminus\{0\}, \diam(\supp(v))\leq s \right\}.
$$
Furthermore, if $F\subseteq X$ and $A\in\mathcal{L}(\lp(X))$, we shall denote the restriction of $A$ to $\ell^p_E(F)$ by $A|_F : \lp(F) \to \lp(X)$. The lower norms $\nu(A|_F)$ and $\nu_s(A|_F)$ shall be understood as the lower norms of $A|_F$ considered as an operator from $\lp(F)$ to $\lp(X)$.
\end{definition}

\begin{remark}\label{rem:low-norm}
Note that if $A$ is an invertible operator, then $\nu(A)=1/\|A^{-1}\|$. Also, $|\nu(A)-\nu(B)|\leq\|A-B\|$ for two operators, $A$ and $B$.
\end{remark}

\begin{theorem}\label{thm:minlimit}
Let $p\in(1,\infty)$ and $E$ be a Banach space.
Assume that $X$ is a space with property A\footnote{In disguise as the metric sparsification property, defined below.}. Let $A\in\AXr$. Then there exists an operator $C\in\opspec(A)$ with $\nu(C)=\inf\{\nu(B)\mid B\in\opspec(A)\}$.
\end{theorem}

The implication (3) $\implies$ (2) in Theorem \ref{main} is an easy corollary of this result, since for an invertible operator $B$, one has $0\not=\nu(B)=1/\|B^{-1}\|$. Hence if all operators in $\opspec(A)$ are invertible, then there is a uniform bound on the norms of their inverses, namely $1/\nu(C)$ from the above statement.

\medskip

For the rest of the section, we fix $p\in(1,\infty)$ and a Banach space $E$.

\subsection{Lower norm localisation}

As alluded to before, for the results in this section, we need a reformulation of property A, called the metric sparsification property. The metric sparsification property was introduced by Chen, Tessera, Wang and Yu in \cite{Chen:2008so} precisely for the purposes of `locally estimating the operator norm'. The gist of the property is that one can choose big sets (in a given measure) that split into well separated uniformly bounded sets.  It does not seem to be obvious that the metric sparsification property is equivalent to property A: this follows on combining results from \cite{Brodzki:2011fk} and \cite{Sako:2012fk}.

We will use the following formulation of the metric sparsification property.  Using \cite[Proposition 3.3]{Chen:2008so} it is equivalent to the official definition (\cite[Definition 3.1]{Chen:2008so}).

\begin{definition}\label{def:MSP}
Let $(X,d)$ be a metric space. Then $X$ has the \emph{metric sparsification property} (MSP) with constant $c\in(0,1]$, if there exists a non-decreasing function $f:\N\to\N$, such that for all $m\in\N$ and any finite positive Borel measure $\mu$ on $X$, there is a Borel subset $\Omega=\sqcup_{i\in I}\Omega_i$ of $X$, such that
\begin{compactitem}
	\item $d(\Omega_i,\Omega_j)\geq m$ whenever $i\not=j\in I$;
	\item $\diam(\Omega_i)\leq f(m)$ for every $i\in I$;
	\item $\mu(\Omega)\geq c\mu(X)$.
\end{compactitem}
\end{definition}

\begin{remark}\label{rem:anyc}
Chen, Tessera Wang and Yu show in \cite[Proposition 3.3]{Chen:2008so} that if a space has the metric sparsification property for some $c\in (0,1]$, then it has it for any $c'\in (0,1)$.  The function $f':\N\to \N$ associated to $c'$ as in the definition of MSP will not be the same as the original $f:\N\to\N$, but inspection of the proof of \cite[Proposition 3.3]{Chen:2008so} shows that $f'$ can be chosen to depend only on $c'$, $c$ and $f$.
\end{remark}

The following proposition is a generalisation of a technical tool of Lindner and Seidel \cite[Proposition 6]{Lindner:2014aa}: they directly prove it for $\Z^N$.  The proof we present here is a straightforward adaptation of the argument that the metric sparsification property implies the operator norm localisation property \cite[Proposition 4.1]{Chen:2008so}. We remark that although the corresponding proof in \cite{Chen:2008so} is formulated for $p=2$ and Hilbert spaces $E$, it works just as well for other exponents $p$ and Banach spaces $E$ with the obvious modifications.

In the spirit of operator norm localisation, we refer to the phenomenon in the proposition as `lower norm localisation', since, roughly speaking, it says that if we fix the propagation ($r$), norm ($M$) and an error ($\delta$), then we can witness the lower norm, up to $\delta$, of any operator (of propagation $r$ and norm $M$) by a vector supported on a set of fixed size.

\begin{proposition}\label{prop:lower-norm}
Let $X$ be a space having the metric sparsification property. For any $\delta>0$, $M\geq0$ and $r\geq 0$ there exists $s\geq 0$, such that
\begin{equation*}
	\nu(A|_F)\leq \nu_s(A|_F) \leq \nu(A|_F)+\delta
\end{equation*}
for any $A\in\mathcal{L}(\lp(X))$ with propagation at most $r$ and norm at most $M$, and any $F\subseteq X$. 

Moreover, the constant $s$ depends only on $M,r,\delta$ and the constant $c>0$ and function $f:\N\to\N$ associated to MSP.  We shall refer to $s$ as the $\nu$-localisation constant (associated to $\delta$, $M$, $r$, and $c$, $f$.).
\end{proposition}

\begin{proof}
The first inequality is trivial. We focus on the second one, in the case when $F=X$ (for the sake of clarity). Fix $r,M\geq 0$ and an operator $A\in\mathcal{L}(\lp(X))$ with propagation at most $r$ and norm at most $M$.  We may assume $r$ is a natural number.

\emph{Step 1}:
Suppose that $v\in\lp(X)\setminus\{0\}$ is such that its support splits into well separated subsets: precisely, $v= \sum_{i\in I}v_i$, $v_i\not=0$ for all $i\in I$ and $d(\supp(v_i),\supp(v_j)) > 2r$ if $i\not=j$. Then
\begin{equation*}
	\frac{\|Av\|}{\|v\|}\geq \inf_{i\in I}\frac{\|Av_i\|}{\|v_i\|}.
\end{equation*}
Indeed, since $A$ can spread the supports of vectors only by at most $r$, the vectors $Av_i$ are still supported on mutually disjoint sets, hence they are mutually orthogonal\footnote{If $p\not=2$, we simply mean that the analogue of the Pythagoras equality holds, i.e.~that $\|Av\|^p = \sum_{i\in I}\|Av_i\|^p$.}. Now suppose the inequality in the above display is false. Then
$$
\|Av\|^p =
\sum_{i\in I}\|Av_i\|^p >
\sum_{i\in I}\frac{\|Av\|^p\|v_i\|^p}{\|v\|^p} =
\frac{\|Av\|^p}{\|v\|^p}\sum_{i\in I}\|v_i\|^p =
\|Av\|^p,
$$
which is a contradiction.

\emph{Step 2}:
Given a vector $w\in\lp(X)\setminus\{0\}$, we show that up to a uniformly estimated modification, we can split its support into well separated, uniformly bounded sets. Indeed, let $\mu$ be the measure on $X$ defined by declaring that the masses of points are $\mu(\{x\})=\|w(x)\|^p$. Using the metric sparsification property as in Remark \ref{rem:anyc}, for any $c\in (0,1)$ there is a function $f:\N\to\N$ such that  there is a subset $\Omega=\sqcup_{i\in I}\Omega_i$ of $X$ such that $d(\Omega_i,\Omega_j)>2r+1$ for $i\neq j$, the diameter of each $\Omega_i$ is at most $f(2r+1)$, and $\mu(\Omega)\geq c\mu(X)$. Note that this means that $\|P_{\Omega}w\|^p=\mu(\Omega)\geq c\mu(X) = c\|w\|^p$.

\emph{Step 3}:
Norm estimates: With $w$ and $\Omega$ as above, we have
\begin{align*}
\|Aw - AP_{\Omega}w\|^p &\leq 
  \|A\|^p\|w - P_{\Omega}w\|^p =
  \|A\|^p\mu(X\setminus\Omega) =
  \|A\|^p(\mu(X)-\mu(\Omega)) \\
&\leq \|A\|^p (1-c)\mu(X) \leq
M^p(1-c)\|w\|^p.
\end{align*}
Consequently, we get
\begin{equation*}
\|AP_{\Omega}w\| \leq \|Aw\| + M(1-c)^{1/p}\|w\|.
\end{equation*}
Since the vector $P_\Omega w$ splits as $P_\Omega w = \sum_{i\in I}P_{\Omega_i}w$ (possibly discarding summands that are $0$), where the summands have $2r+1>2r$ separated supports, we can combine all of the above to obtain
\begin{align*}
	\inf_{i\in I}\frac{\|A(P_{\Omega_i}w)\|}{\|P_{\Omega_i}w\|} \leq
	\frac{\|AP_{\Omega}w\|}{\|P_{\Omega}w\|} \leq
	\frac{\|AP_{\Omega}w\|}{c^{1/p}\|w\|} \leq
	\frac1{c^{1/p}}\frac{\|Aw\|}{\|w\|} + M\left(\frac{1-c}c\right)^{1/p}.
\end{align*}
Also recall that $\diam(\supp(P_{\Omega_i}w))\leq f(2r+1)$. Hence we see that by choosing the vectors $w\not=0$ such that $\frac{\|Aw\|}{\|w\|}$ are arbitrarily close to $\nu(A)$, we can produce a another vector $v\not=0$ (one of the vectors $P_{\Omega_i}w$) whose support has diameter at most $f(2r+1)$ and the fraction $\frac{\|Av\|}{\|v\|}$ is thus arbitrarily close to $\frac{\nu(A)}{c^{1/p}} + M(\frac{1-c}c)^{1/p}$. Thus
\begin{equation*}
 \nu_{f(2r+1)}(A)\leq\frac{\nu(A)}{c^{1/p}} + M(\tfrac{1-c}c)^{1/p}.
\end{equation*} 

\emph{Step 4}:
Replace the error by $\delta$ using that $c$ can be chosen arbitrarily close to $1$ as in Remark \ref{rem:anyc}.

Since $0\leq \nu(A)\leq \|A\|\leq M$, for any $\delta>0$ we can find $c'\in(0,1)$, such that $\frac{\nu(A)}{c^{1/p}} + M(\tfrac{1-c}c)^{1/p}\leq \nu(A)+\delta$, since $c^{1/p}\to 1$ and $M(\frac{1-c}c)^{1/p}\to 0$ as $c\to 1$.  We may thus set $s=f'(2r+1)$, where $f':\N\to \N$ is the functions as in the definition of MSP for $c'$.  Noting that $c'$ depends on $M$ and $\delta$ and using Remark \ref{rem:anyc} we see that $s$ depends only on $r,M,\delta$ and the original parameters $c,f$ associated to MSP.

\emph{Step 5}: Incorporate the restrictions to $F\subseteq X$. The presented proof works exactly the same way, with the same constants, we only need to restrict the supports of the vectors $w$ to the given $F$.
\end{proof}

\begin{remark}
If $X$ has asymptotic dimension at most $d$ (see e.g. \cite{Roe:2003rw} or \cite{Nowak:2012aa}), then it is easily seen to have the metric sparsification property with $c=\frac1{d+1}$.   Quite often one also knows the function $f$ associated with this $c$, the above proof (together with \cite[Proposition 3.3]{Chen:2008so}) makes it possible to be very explicit about the support bound $s$ in these cases. This is in particular true for $\Z^N$ (with $d=N$).  We leave the computation to the reader as an exercise.
\end{remark}

\begin{remark}
Proposition \ref{prop:lower-norm} fails for any space $X$ that does not have the metric sparsification property. By \cite{Sako:2012fk} (see also \cite{Brodzki:2011fk}), this is equivalent to not having the operator norm localisation property, and so \cite[Lemma 4.2]{Roe:2013rt} provides us with $r>0$, $\kappa<1$, a sequence of disjoint finite subsets $X_n$ of $X$, a sequence of positive, norm one operators $A_n\in\mathcal{L}(\ell^2 X_n)$ with propagation at most $r$ and an increasing sequence of positive reals $s_n$ tending to infinity, such that for any $v\in\ell^2X_n$ of norm one, with support of diameter at most $s_n$, one has $\|A_nv\|\leq \kappa$. Furthermore, it is argued in \cite[Proof of Theorem 1.3]{Roe:2013rt} that there are eigenvectors of $A_n$ with eigenvalue $1$.

Taking $N\geq 0$, and denoting $V_n=1-A_n\in\mathcal{L}(\ell^2X_n)$, we see that $\nu_{s_N}(V_n)\geq 1-\kappa$ for all $n\geq N$, so the block-diagonal operator
$$
Q_N=1\oplus_{n\geq N}V_n\in\mathcal{L}\left(\ell^2(X\setminus \sqcup_{n\geq N}X_n)\oplus_{n\geq N}\ell^2X_n\right)
$$
satisfies $\nu_{s_n}(Q_N)\geq 1-\kappa>0$. However $Q_N$ has a non-trivial kernel (as each $V_n$ does), thus $\nu(Q_N)=0$. Observe also that $Q_N$ has norm one and propagation at most $R$. Thus if we choose $0<\delta<1-\kappa$, for any $s>0$ we can take $N$ sufficiently large, so that $s_N>s$ and so the operator $Q_N$ will satisfy $\nu_s(Q_N)\geq \nu_{s_N}(Q_N)\geq 1-\kappa>0+\delta = \nu(Q_N)+\delta$. This violates Proposition \ref{prop:lower-norm}. 

Finally, we note that we can construct suitable operators $V_n$ explicitly under the slightly stronger assumption that $X$ contains a disjoint union of finite subsets $X_n$, such that $\sqcup_nX_n$ is not uniformly locally amenable \cite{Brodzki:2011fk} (in particular, if $\sqcup_nX_n$ is an expander). Namely, we can take Laplacians $V_n=\Delta^{(n)}_r\in\mathcal{L}(\ell^2X_n)$ on (a suitable) scale $r$, defined by
$$
\Delta_r(\delta_x) = \sum_{y\in X_n, d(y,x)\leq r}(\delta_x-\delta_y),\qquad x\in X_n,
$$
see \cite[Section 3]{Roe:2013rt}.
\end{remark}

In preparation for consideration of limit operators, we turn our attention to limit spaces.

\begin{lemma}\label{lem:MSPuniformly}
Assume that a space $X$ has the metric sparsification property for the constant $c$ and associated function $f:\N\to\N$. Then any $X(\omega)$, $\omega\in\partial X$, has the metric sparsification property for the same $c$ and $f$. 

In other words, the family of metric spaces $\{X(\omega)\mid \omega\in \beta X\}$ has the \emph{uniform} metric sparsification property, as defined in \cite[Definition 3.4]{Chen:2008so}.	
\end{lemma}

\begin{proof}
First, recall that for any $\omega\in\partial X$, we can find an isometric copy of any finite set $F\subseteq X(\omega)$ inside $X$, by Proposition \ref{ultequiv}.

Turning our attention to the Definition \ref{def:MSP} of the metric sparsification property, observe that we can demand that the measures $\mu$ appearing in the definition are finitely supported probability measures: given any finite positive Borel measure $\mu$ on $X$, we can rescale it to achieve $\mu(X)=1$ without changing the outcome; and use an approximation argument to get finite support.

Additionally, we can also assume that the set $\Omega$ appearing in the definition of MSP is contained in the support of $\mu$.

We now argue that $X(\omega)$ has MSP with the same parameters as $X$. Given any finitely supported probability measure $\mu$ on $X(\omega)$, Proposition \ref{ultequiv} gives an isometric copy $F\subseteq X$ of its support, $\supp(\mu)$, inside $X$. Hence, we can pull back $\mu$ to $F$, apply the metric sparsification property for $X$ and then push the data back to $\supp(\mu)\subseteq X(\omega)$, providing the required set $\Omega$ and its decomposition in $X(\omega)$, satisfying exactly the same inequalities as in $X$.
\end{proof}

The next step is to prove that if we fix a rich, band-dominated operator $A$, then all its limit operators satisfy the lower norm localisation with the same parameters (cf.~\cite[Corollary 7]{Lindner:2014aa}).

\begin{corollary}\label{lem:LNL-opspec-uniform}
Assume that $X$ has the metric sparsification property and $A$ is a band-dominated operator on $\ell^p_E(X)$. Then for every $\delta>0$ there exists $s\in\N$ such that
$$
\nu(\Phi_\omega(A)|_F)\leq\nu_s(\Phi_\omega(A)|_F)\leq\nu(\Phi_\omega(A)|_F)+\delta
$$
for all $\omega\in\partial X$, such that $A$ is rich at $\omega$, and all $F\subseteq X(\omega)$.
\end{corollary}

\begin{proof}
For brevity, denote by $\$(A)\subseteq \partial X$ the set of all $\omega\in\partial X$ such that $A$ is rich at $\omega$.

We start by observing that any limit operator of $A$ can be approximated by band operators at least as well as $A$ itself.  Indeed, for any $\delta>0$ there exists a band operator $C$ with propagation (some) $r$ and $\|A-C\|<\frac\delta3$. By Theorem \ref{denserich}, we can assume that $C$ is rich at all $\omega\in\$(A)$. Then, for any such $\omega$, we have $\|\Phi_\omega(A)-\Phi_\omega(C)\|<\frac\delta3$ and $\Phi_\omega(C)$ has also propagation at most $r$; this is a consequence of Theorem \ref{richhom}. Also note that $\|\Phi_\omega(C)\|\leq\|\Phi_\omega(A)\|+\frac\delta3\leq\|A\|+\frac\delta3$.

Lemma \ref{lem:MSPuniformly} and Proposition \ref{prop:lower-norm} imply that there is some $s\in \N$ such that 
$$
\nu(B|_F)\leq\nu_s(B|_F)\leq\nu(B|_F)+\delta
$$ 
for all operators $B$ on $\lp(X(\omega))$ with propagation at most $r$ and $\|B\|\leq\|A\|+\frac\delta3$.  Moreover, $s$ does not depend on $\omega\in\partial X$. In particular, this applies when $B=\Phi_\omega(C)$.

Since norm-close operators have close lower norms (Remark \ref{rem:low-norm}), the conclusion of Proposition \ref{prop:lower-norm} yields
$$
\nu_s(\Phi_\omega(A)|_F)\leq\nu_s(\Phi_\omega(C)|_F)+\tfrac\delta3 \leq \nu(\Phi_\omega(C)|_F) +\tfrac{2\delta}3 \leq \nu(\Phi_\omega(A)|_F)+\delta,
$$
for any $\omega\in\$(A)$ and any $F\subseteq X(\omega)$. This proves the second of the required inequalities. The remaining one is obvious.
\end{proof}

\subsection{Lower norm and limit operators}

Recall now that the \emph{operator spectrum} of a rich band-dominated operator $A$ is
$$
\opspec(A):=\{\Phi_\omega(A)~|~\omega\in \partial X\}. 
$$
Before we embark on the proof of Theorem \ref{thm:minlimit}, we need to generalise some facts about the classical operator spectrum for operators on $\Z^N$ (or other groups). As noted before, when $X$ is a discrete group every $X(\omega)$ canonically identifies with $X$.  Thus one can consider all limit operators as living on the same space, namely $\lp(X)$. In this situation, it is known that the set $\opspec(A)$, for $A\in\AXr$, is $\mathcal{P}$-strongly compact. The next lemma is an easy corollary of this compactness fact in the group case; in our setting, where each limit operator lives on its own space $\lp(X(\omega))$, we give a direct proof.

\begin{lemma}\label{lem:limops-compact}
Let $A\in\AXr$.  Let $(\omega_n)$ be a sequence in $\partial X$.  Then there exists a point $\omega\in \beta X$, and a subsequence $(\omega_{n_k})$ such that 
$$
\lim_{k\to\infty}\nu(\Phi_{\omega_{n_k}}(A)|_{B(\omega_{n_k};r)})=\nu(\Phi_\omega(A)|_{B(\omega;r)})
$$
for any $r\geq 0$.
\end{lemma}

\begin{proof}
Fix $(\omega_n)$.  Fix a non-principal ultrafilter $\mu$ on $\N$, and set $\omega:=\lim_{n\to\mu} \omega_n$; as $\beta X$ is compact, $\omega$ is well-defined.  Using a diagonal argument, it will suffice to show that if we fix any subset $M_0$ of $\N$ for which $\mu(M_0)=1$, then for any fixed $r\geq 0$ and $\epsilon>0$ there exists an infinite subset $M$ of $M_0$ such that  
$$
|\nu(\Phi_{\omega_{n}}(A)|_{B(\omega_{n};r)})-\nu(\Phi_\omega(A)|_{B(\omega;r)})|<\epsilon
$$
for all $n\in M$. Note that using Theorem \ref{denserich}, Theorem \ref{richhom} part \eqref{cont}, and Remark \ref{rem:low-norm}, we may assume that $A$ is a band-operator; we will do this from now on.  

Let $\{t_\alpha:D_\alpha\to R_\alpha\}_{\alpha\in X(\omega)}$ be a compatible family for $\omega$.  Set $F:=B(\omega;r+\text{prop(A)})$.  To say that $\omega=\lim_{n\to \mu}\omega_n$ is equivalent to saying that for each $S\subseteq X$ with $\omega(S)=1$, we have that
$$
\mu(\{n\in M_0~|~\omega_n(S)=1\})=1.
$$
is infinite.  Hence in particular, if $M_\alpha:=\{n\in \N~|~\omega_n(D_\alpha)=1\}$, then $\mu(M_\alpha)=1$ and thus if $M_1:=\cap_{\alpha\in F} M_\alpha$, we have $\mu(M_1)=1$.  Note that $t_\alpha$ is compatible with $\omega_n$ for every $n\in M_1$.

Now, using Proposition \ref{concretelimop} there exists a system of local coordinates $\{f_y:F\to G(y)\}_{y\in Y}$ for $F$ such that the associated linear isometries $U:\lp(F)\to \lp(G(y))$ satisfy 
\begin{equation}\label{move in}
\|U^{-1}P_{G(y)}AP_{G(y)}U-P_F\Phi_\omega(A)P_F\|<\epsilon/2.
\end{equation}
As in Remark \ref{rem:local-coords-balls}, we may further assume that each $G(y)$ is equal to the ball $B(y;r+\text{prop}(A))$.
Set $M_2:=\{n\in M_0~|~\omega_n(Y)=1\}$, and note that $\mu(M_2)=1$.  Set $M=M_1\cap M_2$, which satisfies $\mu(M)=1$; we claim this has the desired properties.  

Indeed, fix $n\in M$.  As $n$ is in $M_1$, $\omega_n$ is compatible with $t_\alpha$ for each $\alpha\in F$, and therefore $t_\alpha(\omega_n)$ makes sense\footnote{It need not equal $\alpha$, however.}.  Write $F_n=\{t_\alpha(\omega_n)\in X(\omega_n)~|~\alpha\in F\}$; by an argument analogous to that of Remark \ref{rem:local-coords-balls}, we may assume that $F_n=B(\omega_n;r+\text{prop}(A))$.  It follows from Proposition \ref{concretelimop} that there exists a subset $Z\subseteq X$ (depending on $n$) and a system of local coordinates $\{f^{(n)}_y:F_n\to G(y)\}_{y\in Z}$ such that the corresponding linear isometries $U_n:\lp(F_n)\to \lp(G(y))$ satisfy
\begin{equation}\label{move in n}
\|U_n^{-1}P_{G(y)}AP_{G(y)}U_n-P_{F_n}\Phi_{\omega_n}(A)P_{F_n}\|<\epsilon/2;
\end{equation}
we may again use the argument of Remark \ref{rem:local-coords-balls} to assume that each $G(y)$ equals $B(y;r+\text{prop}(A))$, and thus there is no confusion between this and our earlier use of the notation $G(y)$.  

As $F=B(\omega;r+\text{prop}(A))$, and as $\text{prop}(A)$ is an upper bound for the propagations of $\Phi_\omega(A)$ by Theorem \ref{richhom} part \eqref{prop}, we see that 
$$
P_{F_n}\Phi_{\omega_{n}}(A)|_{B(\omega_{n};r)}P_{F_n}=\Phi_{\omega_{n}}(A)|_{B(\omega_{n};r)};
$$
and analogously using that $F_n=B(\omega_n;r+\text{prop}(A))$, we see that
$$
P_{F}\Phi_{\omega}(A)|_{B(\omega;r)}P_{F}=\Phi_{\omega}(A)|_{B(\omega;r)}.
$$
We thus get
\begin{align}\label{step1}
|\nu(\Phi_{\omega_{n}}(A)|_{B(\omega_{n};r)})&-\nu(\Phi_\omega(A)|_{B(\omega;r)})| \nonumber \\
&=|\nu(P_{F_n}\Phi_{\omega_{n}}(A)|_{B(\omega_{n};r)}P_{F_n})-\nu(P_{F}\Phi_{\omega}(A)|_{B(\omega;r)}P_{F})|. 
\end{align}
As $n$ is in $M_2$, we have $\omega_n(Y)=1$, and thus $\omega_n(Y\cap Z)=1$ and in particular $Y\cap Z\neq \varnothing$.  Moreover for any $y\in Y\cap Z$, we have
$$
\nu(U^{-1}_nP_{G(y)}A|_{B(y;r)}P_{G(y)}U_n)=\nu(U^{-1}P_{G(y)}A|_{B(y;r)}P_{G(y)}U)
$$
and thus line \eqref{step1} is bounded above by
\begin{align*}
 |\nu(P_{F_n}\Phi_{\omega_{n}}(A)|_{B(\omega_{n};r)}&P_{F_n})-\nu(U^{-1}_nP_{G(y)}A|_{B(y;r)}P_{G(y)}U_n)|\\& ~~~~+|\nu(P_{F}\Phi_{\omega_{n}}(A)|_{B(\omega;r)}P_{F})-\nu(U^{-1}P_{G(y)}A|_{B(y;r)}P_{G(y)}U)|
\end{align*}
Now, using Remark \ref{rem:low-norm}, we have that this is bounded above by
\begin{align*}
\|P_{F_n}\Phi_{\omega_{n}}(A)|_{B(\omega_{n};r)}&P_{F_n}-U^{-1}_nP_{G(y)}A|_{B(y;r)}P_{G(y)}U_n^{-1}\| \\ & ~~~~~+\|P_{F}\Phi_{\omega_{n}}(A)|_{B(\omega;r)}P_{F}-U^{-1}P_{G(y)}A|_{B(y;r)}P_{G(y)}U\| 
\end{align*}
and finally the fact that $U$ and $U_n$ are induced by isometries that take $\omega$ and $\omega_n$ respectively to $y$ bounds this above by
$$
\|P_{F_n}\Phi_{\omega_{n}}(A)P_{F_n}-U^{-1}_nP_{G(y)}AP_{G(y)}U_n\|+\|P_{F}\Phi_{\omega_{n}}(A)P_{F}-U^{-1}P_{G(y)}AP_{G(y)}U\|;
$$
lines \eqref{move in} and \eqref{move in n} imply this is less than $\epsilon$ as required.

\end{proof}

\begin{lemma}\label{lem:shifts}
  Let $\omega\in\partial X$ and $\alpha\in X(\omega)$. Then $\lp(X(\omega))=\lp(X(\alpha))$ and for any $A\in\AXr$, the corresponding limit operators are the same, i.e.~$\Phi_\omega(A) = \Phi_\alpha(A)$.
\end{lemma}

\begin{proof}
By Proposition \ref{prop:Xomega-same}, we have $X(\omega) = X(\alpha)$, and the only distinguishing feature is the choice of the basepoint. Hence, $\lp(X(\omega))=\lp(X(\alpha))$.

Let $\{t_\gamma\}_{\gamma\in X(\omega)}$ be a compatible family of partial translations for $\omega$.  It follows from Remark \ref{rem:comp-alt} that $\{t_\gamma\circ t_\alpha^{-1}\}_{\gamma\in X(\omega)}$ is a compatible family of partial translations for $\alpha$. Using Remark \ref{rem:comp-alt}, we compute for $\beta,\gamma \in X(\omega)=X(\alpha)$:
$$
(\Phi_\alpha(A))_{\beta\gamma} = \lim_{x\to\alpha} A_{t_\beta\circ t_\alpha^{-1}(x),t_\gamma\circ t_\alpha^{-1}(x)} = \lim_{x\to\omega}A_{t_\beta(x)t_\gamma(x)} = (\Phi_\omega(A))_{\beta\gamma},
$$
which finishes the proof.
\end{proof}

Before the proof of Theorem \ref{thm:minlimit}, it might be helpful for readers familiar with the work of Lindner and Seidel to discuss the notion of \emph{shift}. (See also Appendix \ref{sec:comparison}.) In the classical band-dominated operator theory the space $X$ is a group, and the group structure gives rise to isometries $V_g\in \mathcal{L}(\ell^p_E(X))$, defined by $(V_gv)(h) = v(g^{-1}h)$. Given an operator $A$ on $\ell^p_E(X)$, the \emph{shift} of $A$ by an element $g\in X$ is the operator $V_g^{-1}AV_g$. If we denote by $e\in X$ the neutral element of the group, changing from $A$ to the shift $V_g^{-1}AV_g$ can be understood also as changing the basepoint $e$ to $g$ and keeping the operator the same. 

Lindner and Seidel in \cite{Lindner:2014aa} apply shifting to limit operators. As explained in Appendix \ref{sec:comparison}, the limit operators in the group case are actually operators on $\ell^p_E(X)$, thus shifting them is possible. Let us explain the analogous operation in the general case, when $X$ is only a metric space. Then the limit operators live on spaces $\ell^p_E(X(\omega))$. Take an operator $B$ on $\ell^p_E(X(\omega))$. The space $X(\omega)$ has a natural basepoint, namely $\omega$. Given $\alpha\in X(\omega)$, the \emph{shift of $B$ to $\alpha$} will be exactly the same operator, $B$, considered now as an operator on $\ell^p_E(X(\alpha))$. 

\begin{proof}[Proof of Theorem \ref{thm:minlimit}]
We shall follow the strategy of the proof of \cite[Theorem 8]{Lindner:2014aa}. So take a sequence of limit operators in $\opspec(A)$, say $B_n=\Phi_{\omega_n}(A)$, such that $\nu(B_n)\to M:=\inf_{B\in\opspec(A)}\nu(B)$. Next, define $\delta_n:=\frac1{2^n}$ and let $(s_n)$ be the sequence of the $\nu$-localisation constants corresponding to $\delta_n$, obtained from applying Lemma \ref{lem:LNL-opspec-uniform}. Without loss of generality, we can assume that $s_{n+1}>2s_n$.

The one remaining technical step is the following, which we postpone for a moment.
\begin{claim}\label{ind step}
For each $n$ there is a point $\alpha_n\in X(\omega_n)$ such that for $m\in \{1,..., n\}$,
$$
\nu(\Phi_{\omega_n}(A)|_{B(\alpha_n;3s_{m})})\leq \nu(\Phi_{\omega_n}(A))+\delta_{m-1}.
$$
\end{claim}

Assuming we have this, Lemma \ref{lem:shifts} implies that we have equalities
$$
\nu(\Phi_{\omega_n}(A)|_{B(\alpha_n;3s_{m})})=\nu(\Phi_{\alpha_n}(A)|_{B(\alpha_n;3s_{m})}) \text{ and } \nu(\Phi_{\omega_n}(A))=\nu(\Phi_{\alpha_n}(A))
$$
whence Claim \ref{ind step} tells us that for all $m$
\begin{equation}\label{ind step alpha}
\nu(\Phi_{\alpha_n}(A)|_{B(\alpha_n;3s_{m})})\leq \nu(\Phi_{\alpha_n}(A))+\delta_{m-1}.
\end{equation}
On the other hand, Lemma \ref{lem:limops-compact} implies there exists $\alpha\in \partial X$ and a subsequence $(\alpha_{n_k})$ of $(\alpha_n)$ such that for all $m$,
$$
\lim_{k\to\infty}\nu (\Phi_{\alpha_{n_k}}(A)|_{B(\alpha_{n_k};3s_{m})})=\nu(\Phi_\alpha(A)|_{B(\alpha;3s_{m})}).
$$
In particular, if we replace $n$ by $n_k$ in line \eqref{ind step alpha} and take the limit as $k$ tends to infinity, then for each $m$ we get
\begin{equation}\label{almost}
\nu(\Phi_\alpha(A|_{B(\alpha;3s_{m})}))=\lim_{k\to\infty}\nu (\Phi_{\alpha_{n_k}}(A)|_{B(\alpha_{n_k};3s_{m})})\leq \liminf_{k\to\infty}\nu(\Phi_{\alpha_{n_k}}(A))+\delta_{m-1}.
\end{equation}
However, Lemma \ref{lem:shifts} again implies that 
$$
\nu(\Phi_{\alpha_{n_k}}(A))=\nu(\Phi_{\omega_{n_k}}(A))
$$ 
for all $k$, and the choice of $(\omega_n)$ implies that $\lim_{k\to\infty}\nu(\Phi_{\omega_{n_k}}(A))$ exists and equals $M$.   Thus line \eqref{almost} implies
$$
\nu(\Phi_\alpha(A|_{B(\alpha;3s_{m})}))\leq M+\delta_{m-1},
$$
which is still valid for all $m$.  Finally, taking the limit as $m$ tends to infinity gives
$$
\nu(\Phi_\alpha(A))\leq M.
$$ 
As $M$ is the infimum of the lower norms of all the limit operators of $A$, however, this forces $\nu(\Phi_\alpha(A))=M$ and we are done. \end{proof}

\begin{proof}[Proof of Claim \ref{ind step}]
Fix $n$.  For notational simplicity, write $B=\Phi_{\omega_n}(A)$ and $\omega=\omega_n$.  We will first construct points $\alpha^{(0)},...,\alpha^{(n)}\in X(\omega)$ and unit vectors $w_0,...,w_n\in \lp(X(\omega))$ with the following properties:
\begin{compactitem}
\item[(i)] for $k\in \{0,...,n\}$ and $i\in \{0,...,k-1\}$, the vector $w_i$ is supported in $B(\alpha^{(k)};2s_{n-i}+s_{n-i-1}+\dots+s_{n-k+1})$ (in particular in $B(\alpha^{(k)};2s_{n-k+1})$ for $i=k-1$);
\item[(ii)]  for $k\in \{0,...,n\}$, $w_k$ is supported in $B(\alpha^{(k)};s_{n-k})$; 
\item[(iii)] for each $i\in \{0,...,k\}$, $\|Bw_i\|<\nu(B)+\delta_n+\delta_{n-1}+\dots+\delta_{n-i}$.
\end{compactitem}

The construction is by induction on $k$.  For the base case $k=0$, Corollary \ref{lem:LNL-opspec-uniform} and the choice of $s_n$ gives us a unit vector $w_0\in \lp(X(\omega))$ supported in some set of diameter at most $s_n$ such that 
$$
\|Bw_0\|<\nu(B)+\delta_n.
$$
Choose $\alpha^{(0)}$ to be any point in $X(\omega)$ such that $w_0$ is supported in $B(\alpha^{(0)};s_n)$.  Clearly $\alpha^{(0)}$ and $w_0$ have the right properties.

For the inductive step, assume we have points $\alpha^{(0)},...,\alpha^{(k)}\in X(\omega)$ and unit vectors $w_0,...,w_k\in \lp(X(\omega))$ satisfying the above properties.  In particular, for $i=k$, we have
\begin{equation}\label{ind ass}
\nu(B|_{B(\alpha^{(k)};s_{n-k})})\leq \|Bw_k\|<\nu(B)+\delta_n+\cdots +\delta_{n-k}.
\end{equation}
Applying Lemma \ref{lem:LNL-opspec-uniform} for the error bound $\delta_{n-(k+1)}$ to $B|_{B(\alpha^{(k)};s_{n-k})}$ gives a unit vector $w_{k+1}\in \lp(B(\alpha^{(k)};s_{n-k}))$ with support in $B(\alpha^{(k+1)};s_{n-k-1})$ for some $\alpha^{(k+1)}\in B(\alpha^{(k)};s_{n-k})$ and with
$$
\|Bw_{k+1}\|<\nu(B|_{B(\alpha^{(k)};s_{n-k})})+\delta_{n-k-1}.
$$ 
Hence from line \eqref{ind ass} we get
$$
\|Bw_{k+1}\|<\nu(B)+\delta_{n}+\cdots +\delta_{n-k}+\delta_{n-k-1}.
$$
This completes the inductive step: indeed, the only remaining condition to check is that $w_0,...,w_k$ satisfy condition (i), and this follows from the inductive hypothesis and the fact that $\alpha^{(k+1)}$ is in $B(\alpha^{(k)},s_{n-k})$.

Now, to complete the proof of the claim, define $\alpha_n=\alpha^{(n)}$.  Then for each $i\in \{0,...,n-1\}$, the vector $w_i$ is supported in 
$$
B(\alpha_n;2s_{n-i}+s_{n-i-1}+\cdots+s_{1}).
$$
As $2s_i<s_{i+1}$, we have $2s_{n-i}+s_{n-i-1}+\cdots s_1\leq 3s_{n-i}$, and so $w_i$ is supported in $B(\alpha_n;3s_{n-i})$.  Moreover, $\delta_n+\delta_{n-1}+\cdots +\delta_{n-i}<\delta_{n-i-1}$, so for each $i\in \{0,...,n-1\}$ we have
$$
\nu(B|_{B(\alpha_n;3s_{n-i})})\leq \|Bw_i\|<\nu(B)+\delta_{n-i-1}.
$$ 
Setting $m=n-i$ for $i\in \{0,...,n-1\}$, we are done.

\end{proof}

\section{Necessity of property A: ghost operators}\label{sec:ghosts}

Throughout this section, $X$ is a space in the sense of Definition \ref{sdbg}, and $E$ is a fixed Banach space.  

For each non-principal ultrafilter $\omega$ on $X$ and $p\in (1,\infty)$, consider the homomorphism
$$
\Phi_\omega:\AXr\to \mathcal{L}(\ell^p_E(X(\omega)))
$$
from Theorem \ref{richhom} which sends each $A\in \AXr$ to its limit operator at $\omega$.  Our goal in this section is to characterise the intersection of the kernels of the various $\Phi_\omega$;  as well as being of some interest in its own right, this allows us to show that property A is necessary for the implication (2) implies (1) of Theorem \ref{main}, at least in the case $p=2$.

We will show that the intersection of the kernels of the homomorphisms $\Phi_\omega$ are the \emph{ghost operators} in the sense of the following definition.  The definition is originally due to Guoliang Yu: compare \cite[Section 11.5.2]{Roe:2003rw}.

\begin{definition}\label{ghostdef}
An operator $A$ in $\AX$ is a \emph{ghost} if for any $\epsilon>0$ there exists a finite subset $F$ of $X$ such that for all $(x,y)\in (X\times X)\setminus (F\times F)$,
$$
\|A_{xy}\|<\epsilon.
$$
\end{definition}

\begin{proposition}\label{ghostlem}
The intersection
$$
\bigcap_{\omega\in \partial X} \kernel(\Phi_\omega)
$$
consists exactly of ghost operators.
\end{proposition}

\begin{proof}
Say first $A$ is a ghost.   Let $\epsilon>0$, and let $F$ be a finite subset of $X$ such that all matrix entries of $A-P_FAP_F$ have norm at most $\epsilon$.  Then for any non-principal ultrafilter $\omega$ in $\beta X$ we have that $\Phi_\omega(A)=\Phi_\omega(A-P_FAP_F)$.  However, all matrix entries of the latter operator are ultralimits over operators of norm at most $\epsilon$, and thus have norm at most $\epsilon$.  As $\epsilon$ was arbitrary, this forces $\Phi_\omega(A)=0$.

Conversely, say $A$ is not a ghost.  Then there exists $\epsilon>0$ and an infinite subset $Y$ of $X\times X$ such that $\|A_{xy}\|\geq \epsilon$ for all $(x,y)\in Y$.  As $A$ is a limit of band-operators, we must have that there exists $R>0$ such that $d(x,y)\leq	R$ for all $(x,y)\in Y$.    Together with the facts that $X$ has bounded geometry, and $Y$ is infinite, this forces the subset $p_1(Y)$ to be infinite, where $p_1:X\times X\to X$ is the projection onto the first coordinate.  Let $\omega\in \beta X$ be any non-principal ultrafilter such that $\omega(p_1(Y))=1$.  As $X$ has bounded geometry, $p_1(Y)$ splits into finitely many disjoint sets $Z_1,...,Z_N$ such that for each $i\in \{1,...,N\}$ and each $z\in Z_i$, there is exactly one $x\in X$ such that $(z,x)$ is in $Y$.  There must exist exactly one $Z_i$ such that $\omega(Z_i)=1$; write $D$ for this $Z_i$.  Define now a partial translation $t$ with domain $D$ by stipulating that for each $z\in X$, $t(z)$ is the unique $x$ such that $(z,x)$ in $Y$.  Clearly $t$ is compatible with $\omega$; write $t(\omega)=\alpha$ for some $\alpha \in \beta X$.  Then the limit 
$$
\lim_{x\to\omega} A_{xt(x)}=A_{\omega\alpha}
$$
is a (norm) ultralimit of operators of norm at least $\epsilon$, and thus has norm at least $\epsilon$.  As it is a matrix coefficient of $\Phi_\omega(A)$, this forces $\Phi_\omega(A)\neq 0$ as required.
\end{proof}

Here is the promised proof that property A is necessary, at least for $p=2$ and $E=\C$.  The assumption $E=\C$ is not really significant, but simplifies the proof.  The assumption that $p=2$ is needed as we rely on results from \cite{Roe:2013rt} which uses operator algebraic techniques; we suspect the result should hold for any $p\in (1,\infty)$, however.

\begin{corollary}\label{aneccesary}
Assume $X$ is a space without property A, $p=2$ and $E=\C$.  Then there exists an operator $A$ in $\AXtwoC$ such that $\Phi_\omega(A)$ is invertible for all non-principal ultrafilters on $X$, and the norms $\|\Phi_\omega(A)^{-1}\|$ are uniformly bounded, but such that $A$ is not Fredholm.
\end{corollary}

\begin{proof}
The main result of \cite{Roe:2013rt} implies that there is a non-compact positive ghost operator $T$ on $\ell^2(X)=\ell^2_\C(X)$.  As $T$ is non-compact and positive there is some non-zero $\lambda$ in the essential spectrum of $T$.  Let $A=\lambda-T$.  Then Proposition \ref{ghostlem} implies that $\Phi_\omega(A)=\lambda$ for any non-principal ultrafilter $\omega$, so certainly all the operators $\Phi_\omega(A)$ are invertible with uniformly bounded inverses.  However, the choice of $\lambda$ guarantees that the essential spectrum of $A$ contains $0$, and so in particular $A$ is not Fredholm.
\end{proof}

\appendix

\section{Conventions on ultrafilters}\label{ultrafilterssec}

The material in this section is all fairly well-known, but we could not find an appropriate reference.  We have thus included it to keep the paper self-contained.

\begin{definition}\label{uf}
Let $A$ be a set and $\mathcal{P}(A)$ its power set.  An \emph{ultrafilter} on $A$ is a function 
$$
\omega:\mathcal{P}(A)\to \{0,1\}
$$
such that $\omega(A)=1$ and so that if a subset $B$ of $A$ can be written as a finite disjoint union $B=B_1\sqcup \cdots \sqcup B_n$ then 
$$
\omega(B)=\sum_{i=1}^n \omega(B_i).
$$
In words: $\omega$ is a finitely additive $\{0,1\}$-valued measure defined on the algebra of all subsets of $A$.

An ultrafilter is \emph{principal} if there exists some $a_0$ in $A$ such that  
$$
\omega(B)=\left\{\begin{array}{ll} 1 & a_0\in B \\ 0 & a_0 \not\in B\end{array}\right.,
$$
for all subsets $B$ of $A$, and is \emph{non-principal} otherwise.
\end{definition}

An argument based on Zorn's lemma shows that (many) non-principal ultrafilters exist whenever $A$ is infinite.

\begin{lemma}\label{limlem}
Let $A$ be a set, and $\omega$ be an ultrafilter on $A$.  Let $D$ be a subset of $A$ such that $\omega(D)=1$, and let $f:D\to X$ be a function from $D$ to a compact Hausdorff topological space $X$.

Then there exists a unique point $x\in X$ such that for any open neighbourhood $U$ of $x$, $\omega(f^{-1}(U))=1$.  
\end{lemma}

\begin{proof}
Uniqueness follows as $X$ is Hausdorff.  Indeed, if $x$ and $y$ were two points with the given property, then there would exist open disjoint neighbourhoods $U\owns x$ and $V\owns y$.  Finite additivity of $\omega$ then implies that
$$
1=\omega(f^{-1}(U\cup V))=\omega(f^{-1}(U))+\omega(f^{-1}(V))=2.
$$

For existence, let $\mathcal{F}$ be the collection of closed subsets $F$ of $X$ such that $\omega(f^{-1}(F))=1$; note that as $\omega(D)=1$, $\mathcal{F}$ contains $X$, so in particular is non-empty.  The collection $\mathcal{F}$ is also closed under finite intersections by finite additivity of $\omega$.  As $X$ is compact, there exists a point $x$ in the intersection $\cap_{F\in \mathcal{F}}F$.  If this $x$ did not have the given property, there would exist an open set $U\owns x$ such that $\omega(f^{-1}(U))=0$.  This forces $\omega(f^{-1}(X\setminus U))=1$, and therefore $x$ is in the closed set $X\setminus U$, which is a contradiction.
\end{proof}

\begin{definition}\label{limdef}
With notation from Lemma \ref{limlem}, the unique point $x$ is denoted $\lim_\omega f$, or $\lim_{a\to\omega}f(a)$ if we want to include the variable.  It is called the \emph{ultralimit of $f$ along $\omega$}, or the \emph{$\omega$-limit} of $f$.
\end{definition}

Note that the special case when $D=A=\N$ reduces to the well-known process of taking an ultralimit over a sequence.  Note also that if we restrict $f$ to a subset $E$ of $D$ such that $\omega(E)=1$, then $\lim_\omega f|_E=\lim_\omega f$.  We will use this fact many times in the body of the paper without further comment.

\begin{definition}\label{stonecech}
Let $A$ be a set. We denote by $\beta A$ the collection of all ultrafilters on $A$.  Note that $A$ identifies canonically with the subset of $\beta A$ consisting of principal ultrafilters.  The set $\beta A$ can be equipped with a topology as follows.

For each subset $B$ of $A$, define
$$
\overline{B}:=\{\omega\in \beta A~|~\omega(B)=1\}.
$$
The topology on $\beta A$ is that generated by the sets $\overline{B}$.  Note that each set $\overline{B}$ is also closed for this topology, and indeed identifies with the closure of the subset $B$ of $\beta A$ (i.e.\ the subset of principal ultrafilters coming from elements in $B$), justifying the notation. 

It is not difficult to check that with this topology $\beta A$ becomes a compact Hausdorff topological space that contains $A$ as a discrete, open, dense subset.  The topological space $\beta A$ is called the \emph{Stone-\v{C}ech compactification of $A$}, and the complement of $A$, denoted $\partial A:=\beta A\setminus A$, is called the \emph{Stone-\v{C}ech corona of $A$}. 
\end{definition}

\section{Comparison with previous definitions of limit operators}\label{sec:comparison}

The notion of limit operators we use in the main part of the paper (see Definition \ref{limopdef} above) looks quite different from those used by by Rabinovich, Roch and Silbermann in \cite[Section 1.2]{Rabinovich:2004qy} and Roe in \cite[Section 2]{Roe:2005lr}.  It is the purpose of this appendix to show that these various notions of limit operator give essentially the same operator spectrum; here `essentially the same' means that the operators appearing in the two spectra are the same up to conjugation by canonical isometric isomorphisms.  In particular, the results of this appendix make it clear that \cite[Theorem 2.2.1]{Rabinovich:2004qy} and \cite[Theorem 3.4]{Roe:2005lr} are equivalent to special cases of Theorem \ref{main}, part (2) if and only  (1), above.  

Throughout this appendix, we work in the following setting.  Let $\Gamma$ be a countable discrete group\footnote{Roe assumes $\Gamma$ is finitely generated, and Rabinovich, Roch and Silbermann that it is $\Z^N$ for some $N$, but these restrictions make no difference to the definitions or proofs.}, and equip $\Gamma$ with any left-invariant\footnote{This means that $d(gh_1,gh_2)=d(h_1,h_2)$ for $g,h_1,h_2$ in $\Gamma$} bounded geometry metric taking a discrete set of values (so in particular, $\Gamma$ equipped with this metric is a space, in the sense of Definition \ref{sdbg} above).  It is well-known that such a metric always exists and any two such metrics are coarsely equivalent: see for example \cite[Proposition 2.3.3]{Willett:2009rt}.  It follows that the algebra of band-dominated operators on $\Gamma$ does not depend on the choice of such a metric.

For each $g\in \Gamma$, let 
$$
\rho_g:\Gamma\to \Gamma,~~~x\mapsto xg
$$
denote the natural right translation map; using left-invariance of the metric, we see that $d(\rho_g(x),x)=d(g,e)$ for all $x$, so in particular each $\rho_g$ is a partial translation.  Each is moreover compatible with every $\omega\in \partial\Gamma$, as each has full domain, and thus for each $g\in \Gamma$ and $\omega\in \beta\Gamma$, the element $\rho_g(\omega)$ is a well-defined element of $\Gamma(\omega)$.  We have the following lemma.

\begin{lemma}\label{limspgp}
For each non-principal ultrafilter on $G$, the natural map
$$
b_\omega:\Gamma\to \Gamma(\omega),~~~g\mapsto \rho_g(\omega)
$$
is an isometric bijection.
\end{lemma}

\begin{proof}
The computation
$$
d_\omega(\rho_x(\omega),\rho_y(\omega))=\lim_{g\to\omega}d(\rho_x(g),\rho_y(g))=\lim_{g\to\omega}d(gx,gy)=d(x,y)
$$
(where the last step is left-invariance) shows that $b_\omega$ is isometric.  To see surjectivity, let $\alpha$ be any element of $\Gamma(\omega)$, and $t:D\to R$ be any partial translation that is compatible with $\omega$ such that $t(\omega)=\alpha$.  For each $g\in \Gamma$, define
$$
D_g:=\{x\in D~|~t(x)=\rho_g(x)\}.
$$  
and note that as $t$ is a partial translation, only finitely many $D_g$ can be non-empty.  It follows that there is a unique $g\in \Gamma$ such that $\omega(D_g)=1$, and it follows that $\alpha=t(\omega)=\rho_g(\omega)$ for this $g$. 
\end{proof}

As usual, we will denote by $\ell^p_E(\Gamma)$ the Banach space of $p$-summable functions on $\Gamma$ with values in a given Banach space $E$.  For each $g\in \Gamma$, let $V_g:\ell^p_E(\Gamma)\to \ell^p_E(\Gamma)$ be the linear and isometric shift operator defined by 
$$
(V_g\xi)(h)=\xi(g^{-1}h).
$$
For each non-principal ultrafilter $\omega$ on $\Gamma$, let $b_\omega:\Gamma\to\Gamma(\omega)$ be the bijective isometry from Lemma \ref{limspgp}, and let $U_\omega:\ell^p_E(\Gamma) \to \ell^p_E(\Gamma(\omega))$ be the corresponding linear isometric isomorphism defined by
$$
(U_\omega\xi)(\alpha)=\xi(b_\omega^{-1}(\alpha)).
$$
This notation will be fixed for the remainder of the appendix.

Having established all this notation, we first look at the definition of limit operator used by Roe \cite[page 413]{Roe:2005lr}.

\begin{definition}\label{roelimop}
Let $\ell^2(\Gamma)$ denote the Hilbert space\footnote{One could also use $p\neq 2$ and auxiliary Banach spaces; however, Roe only considers the case in \cite{Roe:2005lr}, so we will also restrict ourselves to this case for the sake of simplicity.} of square summable functions on $\Gamma$. 

Let $A$ be a band-dominated operator on $\ell^2(\Gamma)$.  Then (\cite[Corollary 2.6]{Roe:2005lr}) the map
$$
\sigma(A):\Gamma\to \mathcal{L}(\ell^2(\Gamma)),~~~g\mapsto V_gAV_g^*
$$
has $*$-strongly precompact range, and so by the universal property of the Stone-\v{C}ech compactification, extends to a $*$-strongly continuous map
$$
\sigma(A):\beta\Gamma\to \mathcal{L}(\ell^2(\Gamma)).
$$
The \emph{Roe limit operator} of $A$ at a non-principal ultrafilter $\omega\in \partial \Gamma$, denoted $A_\omega$, is the image $\sigma(A)(\omega)$.  The \emph{Roe operator spectrum} of $A$, denoted $\opspecRoe(A)$, is the collection $\{A_\omega~|~\omega\in  \partial \Gamma\}$.
\end{definition}

We now show that the Roe operator spectrum is essentially the same as ours.  Note that as $E=\C$ is finite dimensional, all band-dominated operators on $\ell^2(\Gamma)$ are rich in the sense of Definition \ref{richdef} (cf.\ Remark \ref{richrem}).  

\begin{proposition}\label{comproe}
If $A$ is a band-dominated operator on $\ell^2(\Gamma)$ and $\omega$ is a non-principal ultrafilter on $\Gamma$, then $A_\omega=U_\omega^*\Phi_\omega(A)U_\omega$.  

In particular, we have an equality of sets of operators 
$$
\opspecRoe(A)=\{U_\omega^* \Phi_\omega(A)U_\omega~|~\omega\in \partial \Gamma\}.
$$
\end{proposition}

\begin{proof}
It suffices to check that the two operators $A_\omega$ and $U_\omega^*\Phi_\omega(A)U_\omega$ have the same matrix entries.  Indeed, for $x,y\in G$, the $(x,y)^{\text{th}}$ matrix coefficient of $(A_\omega)$ is
$$
\lim_{g\to \omega}(V_gAV_g^*)_{x~y}=\lim_{g\to\omega}A_{gx~gy}. 
$$
On the other hand, the $(x,y)^\text{th}$ matrix coefficient of $U_\omega^*\Phi_\omega(A)U_\omega$ is (by definition of $U_\omega$ in terms of the map $g\mapsto \rho_g(\omega)$) the same as the $(\rho_x(\omega),\rho_y(\omega))^{\text{th}}$ matrix coefficient of $\Phi_\omega(A)$.  Recall that for any $x\in \Gamma$, $\rho_x$ is itself a partial translation on $\Gamma$ (with full domain and codomain) that is compatible with $\omega$ and takes $\omega$ to $\rho_x(\omega)$.  Hence by definition of $\Phi_\omega(A)$, the $(\rho_x(\omega),\rho_y(\omega))^{\text{th}}$ matrix coefficient of this operator is 
$$
\lim_{g\to\omega} A_{\rho_x(g)~\rho_y(g)}=\lim_{g\to\omega} A_{gx~gy},
$$
so we are done.
\end{proof}

We now look at the definition of limit operators used by Rabinovich, Roch and Silbermann \cite[Definition 1.2.1]{Rabinovich:2004qy}, as it applies to the sort of band-dominated operators we consider.   As we mentioned before, Rabinovich, Roch and Silbermann only consider $\Gamma=\Z^N$, but there is no real additional complexity in case of a more general group, so we consider that here.

\begin{definition}\label{rrslimop}
Let $p$ be a number in $(1,\infty)$ and $E$ be a fixed Banach space.  As usual, if $F$ is a subset of $\Gamma$, then we let $P_F$ denote the projection operator on $\ell^p_E(\Gamma)$ corresponding to the characteristic function of $F$.  

Let $\mathbf{h}=(h_n)$ be a sequence in $\Gamma$ that tends to infinity, and let $A$ be band-dominated operator on $\ell^p_E(\Gamma)$.  An operator $A_{\mathbf{h}}$ on $\ell^p_E(\Gamma)$ is the \emph{RRS limit operator} of $A$ with respect to $\mathbf{h}$  if for every finite subset $F$ of $\Z^N$, we have that 
$$
\|P_F(V_{h_n^{-1}}AV_{h_n}-A_{\mathbf{h}})\| \text{ and } \|(V_{h_n^{-1}}AV_{h_n}-A_{\mathbf{h}})P_F\|
$$
tend to zero as $n$ tends to infinity.  The operator $A$ is \emph{RRS rich} (\cite[Definition 1.2.5]{Rabinovich:2004qy}) if for every sequence $\mathbf{h}=(h_n)$ tending to infinity in $\Gamma$, there is a subsequence $\mathbf{h}'=(h_{n_k})$ for which the RRS limit operator $A_{\mathbf{h}'}$ exists.  

The \emph{RRS operator spectrum} of $A$, denoted $\opspecRRS(A)$, is the collection of all RRS limit operators for $A$.
\end{definition}

We start with a lemma giving a weak criterion for recognising RRS limit operators. 

\begin{lemma}\label{rrslimoplem}
Let $B$ and $A$ be band-dominated operators, and let $\mathbf{h}=(h_n)$ be a sequence in $\Gamma$ that tends to infinity.  Assume that for every finite subset $F$ of $\Gamma$ there exists a subsequence $(h_{n_k})$ of $\mathbf{h}$ such that 
$$
\lim_{k\to\infty}\|P_F(V_{h_{n_k}^{-1}}AV_{h_{n_k}}-B)P_F\|=0.
$$
Then there exists a subsequence $\mathbf{h}'$ of $\mathbf{h}$ such that the RRS limit operator $A_{\mathbf{h}'}$ exists and equals $B$.
\end{lemma}

\begin{proof}
Using a diagonal argument and that $\Gamma$ is countable, we may assume that there exists a subsequence of $(h_{n_k})$ such that for all finite $F$
\begin{equation}\label{both sides}
\lim_{k\to\infty}\|P_F(V_{h_{n_k}^{-1}}AV_{h_{n_k}}-B)P_F\|=0.
\end{equation}
Now, consider the sequence $(P_F(V_{h_{n_k}^{-1}}AV_{h_{n_k}}))_{k=1}^\infty$ for some fixed finite set $F$.  For given $\epsilon>0$, let $A'$ be any band operator such that $\|A-A'\|<\epsilon/2$, and let $r$ be the propagation of $A'$.  Then for any $k$,
\begin{align*}
\|P_F& (V_{h_{n_k}^{-1}}AV_{h_{n_k}})-P_F(V_{h_{n_k}^{-1}}AV_{h_{n_k}})P_{N_r(F)}\| \\ \leq & \|P_F(V_{h_{n_k}^{-1}}(A-A')V_{h_{n_k}})\|
+\|P_F(V_{h_{n_k}^{-1}}A'V_{h_{n_k}})P_{N_r(F)}-P_F(V_{h_{n_k}^{-1}}A'V_{h_{n_k}})\|\\ &+\|P_F(V_{h_{n_k}^{-1}}(A-A')V_{h_{n_k}}P_{N_r(F)})\| \\ < & \epsilon.
\end{align*}
Let $\epsilon>0$, and let $r$ be as above for $\epsilon/3$.  Fix a finite subset $F$ of $\Gamma$.  Let $K$ be such that for all $m,k\geq K$,
$$
\|P_{N_r(F)}(V_{h_{n_k}^{-1}}AV_{h_{n_k}}-V_{h_{n_m}^{-1}}AV_{h_{n_m}})P_{N_r(F)}\|<\epsilon/3
$$
(which exists by our assumption).  Hence for any $k,m\geq K$,
\begin{align*}
\|P_{F} & (V_{h_{n_k}^{-1}}AV_{h_{n_k}}-V_{h_{n_m}^{-1}}AV_{h_{n_m}})\| \\
& <2\epsilon/3+\|P_{N_r(F)}(V_{h_{n_k}^{-1}}AV_{h_{n_k}}-V_{h_{n_m}^{-1}}AV_{h_{n_m}})P_{N_r(F)}\| \\
& <\epsilon.
\end{align*}
Hence for any finite subset $F$ of $\Gamma$, the sequence $(P_F(V_{h_{n_k}^{-1}}AV_{h_{n_k}}))_k$ is Cauchy, and so convergent.  Similarly, for any $F$, the sequence $((V_{h_{n_k}^{-1}}AV_{h_{n_k}})P_F)_k$ is Cauchy so convergent.  Checking matrix entries using line \eqref{both sides}, we must have that these two sequences converge to $P_FB$ and $BP_F$.  Looking back at the definition of RRS limit operator, this completes the proof.
\end{proof}

\begin{proposition}\label{rrsrich}
Fix $p\in(0,1)$ and a Banach space $E$. Then a band-dominated operator $A$ on $\ell^p_E(X)$ is RRS rich if and only if it is rich in the sense of Definition \ref{richdef} above. Moreover
$$
\opspecRRS(A)=\{U_\omega^{-1}\Phi_\omega(A)U_\omega~|~\omega\in \partial \Gamma\}.
$$
\end{proposition}

\begin{proof}
We first summarise some formulas for $(x,y)^\text{th}$ matrix entries of limit operators. If $A_{\mathbf{h}}$ is a RRS limit operator of $A$ for some sequence $\mathbf{h}=(h_n)_{n\in\N}$, then
\begin{equation}\label{eq:rrs-matrix-entry}
  (A_{\mathbf{h}})_{x~y} = \lim_{n\to\infty}\left(V_{h_{n}^{-1}}AV_{h_{n}}\right)_{x~y}=
  \lim_{n\to\infty}A_{h_nx~h_ny} =
    \lim_{n\to\infty}A_{\rho_x(h_{n})~\rho_y(h_{n})}.
\end{equation}
On the other hand, if $\Phi_\omega(A)$ is the limit operator in the sense of Definition \ref{richdef}, then by our definitions and Lemma \ref{limspgp}
\begin{equation}\label{eq:limop-matrix-entry}
  \left(U_\omega^{-1}\Phi_\omega(A)U_\omega\right)_{x~y} = \lim_{g\to\omega} A_{\rho_x(g)~\rho_y(g)}= \lim_{g\to\omega}A_{xg~yg}.
\end{equation}
Furthermore, if $\omega\in\partial\Gamma$ is such that $\omega(\{h_n\mid n\in\N\})=1$ and the limit \eqref{eq:rrs-matrix-entry} exists, then the limit \eqref{eq:limop-matrix-entry} exists, and the two limits are equal.

Embarking on the proof of Proposition,
assume first that $A$ is an RRS rich band-dominated operator.  To see that $A$ is rich in the sense of Definition \ref{richdef} it suffices to show that for any $x,y\in \Gamma$, the limit \eqref{eq:limop-matrix-entry} exists (for the norm topology on $\mathcal{L}(E)$).  Note, however, that the condition of RRS richness implies that for any $x,y\in \Gamma$, any sequence in the set
$$
\{(V_g^{-1}AV_g)_{x~y}\in \mathcal{L}(E)~|~g\in \Gamma\}
$$
has a convergent subsequence (with limit possibly not in the set) for the norm topology on $\mathcal{L}(E)$.  However, this set is precisely equal to
$$
\{A_{xg~yg}~|~g\in \Gamma\}
$$ 
and thus we may conclude that this set is norm precompact.  Hence the limit in line \eqref{eq:limop-matrix-entry} exists by the universal property of $\beta\Gamma$.

We now prove that if $A_{\mathbf{h}}\in \opspecRRS(A)$ for some sequence $\mathbf{h}=(h_n)_{n\in\N}$, then there exists $\omega\in\partial\Gamma$, such that $A_{\mathbf{h}}=U_\omega^{-1}\Phi_\omega(A)U_\omega$. Take any $\omega\in\partial\Gamma$, such that $\omega(\{h_n\mid n\in\N\})=1$. We already know that $\Phi_\omega(A)$ exists, so by the observation at the beginning of the proof, the matrix coefficients of $U_\omega^{-1}\Phi_\omega(A)U_\omega$ are the same as the matrix coefficients of $A_{\mathbf{h}}$, hence the operators are the same. Summarising, $$
\opspecRRS(A)\subseteq\{U_\omega^{-1}\Phi_\omega(A)U_\omega~|~\omega\in \partial \Gamma\}.
$$

For the converse assume that $A$ is rich in the sense of Definition \ref{richdef}. Let $\mathbf{h}=(h_n)$ be a sequence tending to infinity in $\Gamma$ and write $H_0=\{h_n~|~n\in \N\}$.  Let $\omega$ be a non-principal ultrafilter on $\Gamma$ such that $\omega(H_0)=1$.

Let $F$ be a finite subset of $\Gamma$. Now, as $A$ is rich, all the limits 
$$
\lim_{g\to\omega} A_{\rho_x(g)~\rho_y(g)}
$$ 
exist as $x,y$ range over $\Gamma$.  As $F$ is finite, there is a subset $H_1$ of $H_0$ such that
$$
\|A_{\rho_x(g)~\rho_y(g)}-A_{\rho_x(h)~\rho_y(h)}\|<2^{-1}
$$
for all $g,h$ in $H_1$ and $x,y\in F$, and so that $\omega(H_1)=1$.  Continuing in this way, we get a nested sequence of subsets 
$$
H_0\supseteq H_1\supseteq H_2\supseteq
$$
of $\Gamma$, all of $\omega$-measure one, such that for all $x,y\in F$ and all $g,h\in H_k$,
$$
\|A_{\rho_x(g)~\rho_y(g)}-A_{\rho_x(h)~\rho_y(h)}\|<2^{-k}.
$$
Choose now any subsequence $(h_{n_k})$ of $\mathbf{h}$ such that $h_{n_k}$ is in $H_k$. Then the sequence
\begin{equation}\label{eq:mat-ent-rrs-seq}
  \left(A_{\rho_x(h_{n_k})~\rho_y(h_{n_k})}\right)_{k=1}^\infty
\end{equation}
is Cauchy in $\mathcal{L}(\ell^p_E(\Gamma))$ and so convergent for all $x,y\in F$. The limit of this sequence is in fact equal to the $(x,y)^{\text{th}}$ matrix entry of $U_\omega^{-1}\Phi_\omega(A)U_\omega$, by our choice of $h_{n_k}$ and the computation \eqref{eq:limop-matrix-entry}. Consequently, by finiteness of $F$, the sequence
$$
\big(P_FV_{h_{n_k}^{-1}}AV_{h_{n_k}}P_F\big)_{k=1}^\infty
$$
converges in $\mathcal{L}(\ell^p_E(\Gamma))$ to $P_FU_\omega^{-1}\Phi_\omega(A)U_\omega P_F$. Using Lemma \ref{rrslimoplem}, we see that $U_\omega^{-1}\Phi_\omega(A)U_\omega$ is the RRS limit operator $A_{\mathbf{h'}}$ for some subsequence $\mathbf{h'}$ of $\mathbf{h}$.

From this consideration, we can draw two conclusions: if $A$ is rich in the sense of Definition \ref{richdef}, then it is also RRS rich, and for any $\omega\in\partial\Gamma$, the operator $U_\omega^{-1}\Phi_\omega(A)U_\omega$ belongs to $\opspecRRS(A)$. This finishes the proof of the Proposition.
\end{proof}

\section{Groupoid $C^*$-algebra approach}\label{sec:groupoids}

In this section, we sketch the connections of our approach to the theory of groupoids and their $C^*$-algebras.  The machinery of groupoid $C^*$-algebras allows a relatively short proof of the part of Theorem \ref{main} showing that condition (2) is equivalent to condition (1), at least in the case of band-dominated operators on $\ell^2_\C(X)$ for a space $X$.  The approach is very similar to Roe's work in the context of discrete groups \cite{Roe:2005lr}.  As this approach through groupoids assumes quite a lot of machinery, it is difficult to argue that it is genuinely `simpler' than the approach in the main body of the paper, but it does provide a short conceptual proof for those readers familiar with the necessary background.

We have not made any effort to keep this material self-contained: a basic reference for what we need from groupoid $C^*$-algebra theory is Renault's notes \cite{Renault:2009zr}.  

Let $G$ be a locally compact, Hausdorff \'{e}tale groupoid with unit space $G^{(0)}$ and source and range maps $r:G\to G^{(0)}$, $s:G\to G^{(0)}$.  For each $x\in G^{(0)}$, write
$$
G_x=s^{-1}(x)=\{g\in G~|~s(g)=x\}
$$
for the source fibre at $x$; as $G$ is assumed \'{e}tale, the topology this inherits from $G$ is the discrete topology.  Let $C_c(G)$ denote the convolution $*$-algebra of compactly supported, complex-valued  continuous functions on $G$.

Any point $x$ in the unit space $G^{(0)}$ gives rise to a \emph{regular representation} (compare \cite[2.3.4]{Renault:2009zr}) 
$$
\pi_x:C_c(G)\to \mathcal{B}(\ell^2(G_x))
$$
defined for $f\in C_c(G)$, $v\in \ell^2(G_x)$ and $g\in G_x$ by
$$
(\pi_x(f)v)(g)=\sum_{h\in G_x}f(gh^{-1})v(h).
$$
It is not difficult to check that $\pi_x$ is a well-defined $*$-homomorphism.  The \emph{reduced norm} on $C_c(G)$ is then defined by 
$$
\|f\|_r:=\sup_{x\in G^{(0)}}\|\pi_x(f)\|,
$$
and the reduced groupoid $C^*$-algebra $C^*_r(G)$ is the completion of the $*$-algebra $C_c(G)$ in this norm.

Let $X$ be a space as in Definition \ref{sdbg}.  Let $G(X)$ be the coarse groupoid on $X$ as introduced by Skandalis, Tu, and Yu in \cite{Skandalis:2002ng} (see also \cite[Chapter 10]{Roe:2003rw}). As a set $G(X)$ identifies with $\cup_{r>0}\overline{E_r}^{\beta X\times \beta X}$, where $E_r$ is defined by
$$
E_r:=\{(x,y)\in X~|~d(x,y)\leq r\}.
$$
The groupoid operations are the restriction of the pair groupoid operations from $\beta X\times \beta X$.  The topology on $G(X)$ agrees with the subspace topology from $\beta X\times\beta X$ on each $\overline{E}_r$, and is globally defined by stipulating that a subset $U$ of $G$ is open if and only if $U\cap \overline{E}_r$ is open in $\overline{E}_r$ for all $r>0$ (this is \emph{not} the subspace topology from $\beta X\times\beta X$!).  Equipped with this topology, $G(X)$ is a locally compact, $\sigma$-compact (not second countable) Hausdorff, \'{e}tale groupoid.  Write $G_\infty(X)$ for the restriction of the coarse groupoid to the closed saturated subset $\partial X$ of $\beta X=G(X)^{(0)}$ and $X\times X$ for the restriction of $G(X)$ to the open saturated subset $X$ of $\beta X$ (which is just the pair groupoid of $X$, with the discrete topology).  Note that this gives a decomposition 
$$
G(X)=X\times X\sqcup G_\infty(X)
$$ 
and a corresponding short exact sequence of convolution algebras
$$
0\to C_c(X\times X) \to C_c(G(X))\to C_c(G_\infty(X))\to 0.
$$
Writing $C^*_r(X\times X)$, $C^*_r(G(X))$, and $C^*_r(G_\infty(X))$ for the reduced groupoid $C^*$-algebras of $X\times X$, $G(X)$ and $G_\infty(X)$ respectively, we may complete this sequence to a sequence of $C^*$-algebras 
\begin{equation}\label{gpdses}
0\to C^*_r(X\times X)\to C^*_r(G(X))\to C^*_r(G_\infty(X))\to 0.
\end{equation}

The following lemma is essentially proved in \cite[Proposition 10.29]{Roe:2003rw}.  

\begin{lemma}\label{gpdbdo}
Let $X$ be a space.  Let $f$ be an element of the convolution $*$-algebra $C_c(G(X))$, so $f$ is a continuous function supported in $\overline{E_r}$ for some $r>0$. We interpret $f$ as a bounded function from $X\times X$ to $\C$ which is zero on the complement of $E_r$. Define an operator $A_f$ on $\ell^2(X)$ by setting its matrix coefficients to be
$$
(A_f)_{xy}=f(x,y);
$$
note that $A_f$ is a band operator as $f$ is supported on $E_r$.
Then the assignment
$$
\Psi:C_c(G(X))\to \AXtwoC,~~~f\mapsto A_f
$$
extends to a $*$-isomorphism of $C^*_r(G(X))$ onto $\AXtwoC$, the $C^*$-algebra of band-dominated operators on  $\ell^2(X)$.  Moreover, $\Psi$ takes the ideal $C^*_r(X\times X)$ onto the compact operators on $\ell^2(X)$.  \qed
\end{lemma}

\begin{definition}\label{gpdlimopdef}
Let $X$ be a space.  Let $f$ be an element of $C^*_r(G(X))$.  For any $\omega\in \beta X$, the associated  \emph{limit operator over $\omega$} is
$$
\pi_\omega(f)\in \mathcal{L}(\ell^2(G(X)_\omega)).
$$
\end{definition}

\begin{lemma}\label{gpdlimopsame}
Let $\omega$ a non-principal ultrafilter on a space $X$.  Define 
$$
F:X(\omega)\to G(X)_\omega,~~~\alpha\mapsto (\alpha,\omega).
$$
Then $F$ is a bijection.  

Moreover, if 
$$
U:\ell^2(G(X)_\omega)\to \ell^2(X(\omega)),~~~(Uv)(\alpha):=v(\alpha,\omega)
$$
is the unitary isomorphism induced by $F$ and  
$$
\Psi:C^*_r(G(X))\to \AXtwoC,~~~\Psi(f)=A_f
$$ 
is the canonical $*$-isomorphism from Lemma \ref{gpdbdo}, then 
$$
U^*\Phi_\omega(\Psi(f))U=\pi_\omega(f)
$$
where $\Phi_\omega(\Psi(f))$ is as in Definition \ref{limopdef} and $\pi_ \omega(f)$ is as in Definition \ref{gpdlimopdef}.
\end{lemma}

\begin{proof}
Note first that $F$ is well-defined as if $(x_\lambda)$ is a net converging to $\omega$ and $t_\alpha$ is a partial translation compatible with $\omega$ such that $t_\alpha(\omega)=\alpha$, then (possibly after passing to a subnet of $x_\lambda$ in the domain of $t_\alpha$) $(t_\alpha(x_\lambda),x_\lambda)$ is a net in some $E_r$ that converges to $(\alpha,\omega)$.  It follows from this that $(\alpha,\omega)$ is in $G(X)_\omega$.  

The map $F$ is clearly injective.  To see surjectivity, take $(\alpha,\omega)\in G(X)_\omega$, say $(\alpha,\omega)\in \overline{E_r}^{\beta X\times \beta X}$ for some $r\geq 0$. Recall \cite[Corollary 10.18]{Roe:2003rw} that the inclusion $E_r\to X\times X$ extends to a homeomorphism
$$
\overline{E_r}^{\beta(X\times X)} \to \overline{E_r}^{\beta X\times \beta X} \subseteq \beta X\times \beta X,
$$
whence we can think of $(\alpha,\beta)$ as an element of $\beta(X\times X)$, i.e.~an ultrafilter on $X\times X$, which assigns $1$ to the set $E_r$. Now decomposing $E_r$ into a finite disjoint union of graphs of partial translations (see for example the proof of Lemma \ref{decomp}) yields a partial translation, say $t: D\to R$, such that
$$
(\alpha,\omega) \in \overline{\left\{(t(x),x) \mid x\in D \right\}}^{\beta(X\times X)}.
$$
Using \cite[discussion in 10.18--10.24]{Roe:2003rw} again, we conclude that $\omega(D)=1$ (so that $t$ is compatible with $\omega$), and that $t(\omega)=\alpha$. Hence $\alpha\in X(\omega)$.

To complete the proof, we compare matrix coefficients.  Take $f\in C_c(G(X))$.  Given $(\alpha,\omega),(\beta,\omega)\in G(X)_\omega$, we obtain
$$
\pi_\omega(f)_{(\alpha,\omega)(\beta,\omega)} = f((\alpha,\omega)\circ(\omega,\beta)) = f((\alpha,\beta)).
$$
Say $t_\alpha$ and $t_\beta$ are compatible with $\omega$ and such that $t_\alpha(\omega)=\alpha$ and $t_\beta(\omega)=\beta$. Then by continuity of $f$, we have
$$
f((\alpha,\beta))=\lim_{x\to \omega}f(t_\alpha(x),t_\beta(x));
$$ 
on the other hand, by definition of $\Psi$,
$$
\lim_{x\to \omega}f(t_\alpha(x),t_\beta(x))=\lim_{x\to \omega} (\Psi(f))_{t_\alpha(x)~t_\beta(x)} = \Phi_\omega(\Psi(f))_{\alpha\beta}.
$$
Putting this together 
$$
\Phi_\omega(\Psi(f))_{\alpha\beta}=\pi_\omega(f)_{(\alpha,\omega)(\beta,\omega)},
$$
which is the desired statement for $f\in C_c(G(X))$; the proof is completed by continuity of $U$.
\end{proof}

\begin{theorem}\label{maingpd}
Say $X$ is a space with property A. Let $\Psi:C^*_r(G(X))\to \AXtwoC$ be the canonical isomorphism from Lemma \ref{gpdbdo}. Let $f$ be an operator in $C^*_r(G(X))$. Then the following are equivalent.
\begin{enumerate}[(1)]
\item $\Psi(f)$ is Fredholm.
\item $f$ is invertible in $C^*_r(G_\infty(X))$.
\item There exists $c>0$ such that the following holds. For each $\omega$ in $\partial X$, the operator $\pi_\omega(f)$ is invertible, and 
$$
\|\pi_\omega(f)^{-1}\|\leq c.
$$
\item There exists $c>0$ such that the following holds. For each $\omega$ in $\partial X$, the operator $\Phi_\omega(\Psi(f))$ is invertible, and 
$$
\|\Phi_\omega(\Psi(f))^{-1}\|\leq c.
$$
\end{enumerate}
\end{theorem}

\begin{proof}
Consider the sequence of $C^*$-algebras 
$$
0\to C^*_r(X\times X)\to C^*_r(G(X))\to C^*_r(G_\infty(X))\to 0
$$
from line \eqref{gpdses} above. 
In general, this need not be exact at the middle term.  However, if $X$ has property A, then the groupoid $G(X)$ is amenable by \cite[Theorem 5.3]{Skandalis:2002ng}; this in turn implies amenability of $G_\infty(X)$.  Moreover, the pair groupoid $X\times X$ is automatically amenable.  Hence by \cite[Corollary 5.6.17]{Brown:2008qy} the maximal and reduced groupoid $C^*$-algebras of these three groupoids are the same.  On the other hand, the sequence 
$$
0\to C^*_{\max}(X\times X)\to C^*_{\max}(G(X))\to C^*_{\max}(G_\infty(X))\to 0
$$
is well-known to be exact automatically: the only issue is that exactness could fail at the middle term, and thus to show that any representation of $C_c(G(X))$ that contains $C_c(X\times X)$ in its kernel, and thus defines a representation of $C_c(G_\infty(X))$,  extends to $C^*_{\max}(G_\infty(X))$; this follows from the universal property of the maximal completion.   

We may thus conclude that the sequence in line \eqref{gpdses} is exact.  The natural identification from Lemma \ref{gpdbdo} of the middle term with $\AXtwoC$ identifies the ideal $C^*_r(X\times X)$ with $\mathcal{K}(\ell^2(X))$, and so the equivalence of parts (1) and (2) follows from exactness of this sequence and Atkinson's theorem.  

The fact that (2) implies (3) follows from the fact that the $*$-homomorphisms
$$
\pi_\omega:C^*_r(G(X))\to \mathcal{L}(\ell^2(G(X)_\omega))
$$ 
are automatically contractive and factor through $C^*_r(G_\infty(X))$ for all $\omega\in \partial X$. To see that (3) implies (2), note that the definition of the reduced norm implies that the direct sum representation
$$
\pi:=\bigoplus_{\omega\in \partial X}\pi_\omega:C^*_r(G_\infty(X))\to \mathcal{L}\Big(\bigoplus_{\omega\in \partial X} \ell^2(G(X)_\omega)\Big)
$$
is faithful.  The condition in (3) guarantees that the operator
$$
B:=\bigoplus_{\omega\in \partial X} (\pi_\omega(f))^{-1}
$$
makes sense on $\oplus \ell^2(G(X)_\omega)$, and it is clearly the inverse of $\pi(f)$.  As $C^*$-algebras are inverse closed, $B$ is in $\pi(C^*_r(G_\infty(X)))$, and whatever operator in $C^*_r(G_\infty(X))$ maps to $B$ under the faithful representation $\pi$ must in fact be an inverse to $f$ in $C^*_r(G_\infty(X))$.

Finally, note that (3) is equivalent to (4) by Lemma \ref{gpdlimopsame}.
\end{proof}

We do not know a short proof that the uniform boundedness condition in Theorem \ref{maingpd} is unnecessary.

\bibliographystyle{amsplain}
\bibliography{Generalbib}

\end{document}